\documentclass[final,onefignum,onetabnum]{siamart171218}
\usepackage{tikz}
\usepackage{amsmath,amssymb}
\usepackage[mathscr]{euscript}
\usepackage{stmaryrd}
\usepackage{mathtools}
\usepackage{algorithmic}
\Crefname{ALC@unique}{Line}{Lines} % <- Preamble

\title{Multilinear Control Systems Theory\thanks{Submitted to the
    editors May 17, 2019. This is an extended version of a conference paper presented in SIAM CT19.
\funding{This work is supported in part under AFOSR Award No: FA9550-18-1-0028, NSF grant DMS 1613819, Smale Institute, and the Guaranteeing AI Robustness Against Deception (GARD) program from DARPA (I2O).}}}
\author{Can Chen\thanks{Department of Mathematics and Department of Electrical Engineering and Computer Science, University of Michigan,
    Ann Arbor, MI 48109 (\email{canc@umich.edu}).}
\and Amit Surana\thanks{Raytheon Technologies Research Center, East Hartford,
    CT 06108 (\email{amit.surana@rtx.com}).}
\and Anthony Bloch\thanks{Department of Mathematics, University of
    Michigan, Ann Arbor, MI 48109 (\email{abloch@umich.edu}).}
\and Indika Rajapakse\thanks{Department of Computational Medicine \&
    Bioinformatics, Medical School and Department of Mathematics,
    University of Michigan, Ann Arbor, MI 48109 (\email{indikar@umich.edu}).}}

\headers{Multilinear Control Systems Theory}{Can Chen, Amit Surana, Anthony Bloch,
    and Indika Rajapakse}

\begin{document}

\maketitle

\begin{abstract}
In this paper, we provide a system theoretic treatment of a new class of multilinear time-invariant (MLTI) systems in which the states, inputs and outputs are tensors, and the system evolution is governed by multilinear operators. The MLTI system representation is based on the Einstein product and even-order paired tensors. There is a particular tensor unfolding which gives rise to an isomorphism from this tensor space to the general linear group, i.e. the group of invertible matrices. By leveraging this unfolding operation, one can extend classical linear time-invariant (LTI) system notions including stability, reachability and observability to MLTI systems. While the unfolding based formulation is a powerful theoretical construct, the computational advantages of MLTI systems can only be fully realized while working with the tensor form, where hidden patterns/structures can be exploited for efficient representations and computations. Along these lines, we establish new results which enable one to express tensor unfolding based stability, reachability and observability criteria in terms of more standard notions of tensor ranks/decompositions. In addition, we develop a generalized CANDECOMP/PARAFAC decomposition and tensor train decomposition based model reduction framework, which can significantly reduce the number of MLTI system parameters. We demonstrate our framework with numerical examples.

\end{abstract}
\begin{keywords}
multilinear time-invariant systems, stability, reachability, observability, model reduction, tensor unfolding, tensor ranks/decompositions, block tensors
\end{keywords}
\begin{AMS}
 15A69, 93B05, 93B07, 93B40, 93C05, 93D20
\end{AMS}

\section{Introduction}\label{sec:intro}
Controlling high-dimensional systems remains an extremely challenging task as many control strategies do not scale well with the dimension of the systems. Of particular interest in this paper are complex biological and engineering systems in which structure, function and dynamics are highly coupled. Such interactions can be naturally and compactly captured by tensors. Tensors are multidimensional arrays generalized from vectors and matrices, and have wide applications in many domains such as social networks, biology, cognitive science, applied mechanics, scientific computation and signal processing \cite{Chen_2015, DING201875, gel2017tensor, tensorSC, Kolda06multilinearoperators}. For example,  the organization of the interphase nucleus in the human genome reflects a dynamical interaction between 3D genome structure, function and its relationship to phenotype, a concept known as the 4D Nucleome (4DN) \cite{Chen_2015}. 4DN research requires a comprehensive view of genome-wide structure, gene expression, the proteome and phenotype, which fits naturally with a tensorial representation \cite{Rajapakse711,7798500}. In order to apply the standard system and control framework in applications such as these, tensors need to be vectorized, leading to an extremely high-dimensional system representation in which the number of states/parameters scale exponentially with the number of dimensions of the tensors involved \cite{7798500}. Moreover, with the vectorization of tensors, hidden patterns/structures, e.g. redundancy/correlations, can get lost, and thus one cannot exploit such inherent structures for efficient representations and computations.

In order to take advantage of tensor algebraic computations, recently a new class of multilinear time-invariant (MLTI) system has been introduced \cite{rogers_2013,7798500}, in which the states and outputs are preserved as tensors. The system evolution is generated by the action of multilinear operators which are formed using Tucker products of matrices. By using tensor unfolding, an operation that transforms a tensor into a matrix, Rogers et al. \cite{rogers_2013} and Surana et al. \cite{7798500} developed methods for model identification/reduction from tensor time series data. An application of such tensor based representation/identification for skeleton based human behavior recognition from videos demonstrated significant improvements in classification accuracy compared to standard linear time-invariant (LTI) based approaches \cite{DING201875}.  However, the MLTI system representation is limited because it assumes the multilinear operators are formed from the Tucker products of matrices (and thus precludes more general tensorial representations) and does not incorporate control inputs.

The role of tensor algebra has also been explored for modeling and simulation of nonlinear dynamics, where the vector field is a multilinear function of states \cite{KRUPPA20175610}. Tensor decomposition techniques such as CANDECOMP/PARAFAC decomposition (CPD) and tensor train decomposition (TTD) can reduce system size, thus reducing memory usage and enabling efficient computation during simulations. Note that in contrast to the MLTI systems framework of \cite{rogers_2013,7798500}, in this application, tensor algebra is applied to the system represented in the conventional vector form. The author in \cite{gel2017tensor} exploits tensor decompositions to compute numerical solutions of master equations associated with Markov processes on extremely large state spaces.  The Einstein product and even-order paired tensors, along with TTD, were utilized for developing tensor representations for operators based on nearest-neighbor interactions, construction of pseudoinverses for dimensionality reduction methods and the approximation of transfer operators of dynamical systems.

Similarly, using the Einstein product and even-order paired tensors, Chen et. al. \cite{Chen_2018} generalized the notion of MLTI systems introduced in \cite{rogers_2013,7798500} and also incorporated control inputs. The Einstein product is a tensor contraction used quite often in tensor calculus and has profound applications in the study of continuum mechanics and the field of relativity theory \cite{Einstein_2007,Lai_2009}. Moreover, the space of even-order tensors with the Einstein product has many desirable properties. Brazell et al. \cite{doi:10.1137/100804577} discovered that one particular tensor unfolding gives rise to an isomorphism from this tensor space (of even-order tensors equipped with the Einstein  product) to the general linear group, i.e. the group of invertible matrices. This isomorphism enables one to define matrix equivalent concepts for tensors including tensor inverse, positive definiteness and eigenvalue decomposition. Using these tensor constructs, Chen et. al. \cite{Chen_2018} developed tensor algebraic conditions for stability, reachability and observability for generalized input/output MLTI systems. A new notion of block tensors was also introduced which enables one to express these conditions in a  compact fashion. Interestingly, these conditions look analogous to the classical conditions for stability, reachability and observability in LTI systems, and reduce to them as a special case.

This paper is an extended version of the introductory paper \cite{Chen_2018}, and in addition to providing various technical details, we also present several new results. The key contributions of this paper are as follows:

\begin{enumerate}
  \item In \cite{Chen_2018}, the reachability and observability conditions for MLTI systems were stated in terms of the unfolding rank which requires matricization of the reachability/observability tensors. Here we establish new results relating the unfolding rank to other more standard notions of tensor ranks, including multilinear ranks, CP rank and TT-ranks. Using such relations, we provide criteria for reachability and observability which do not require tensor unfolding, and can be computed using efficient tensor algebraic methods. Similarly, we express MLTI system stability conditions using higher-order singular value decomposition (HOSVD), CPD and TTD.
  \item Using generalized CPD/TTD, we develop a framework for model reduction of MLTI systems. This approach takes advantage of tensor decompositions which otherwise cannot be exploited after unfolding the MLTI systems to obtain a standard LTI form. It also successfully realizes the tensor decomposition based criteria for stability, reachability and observability. Furthermore, we establish new stability results by utilizing the factor matrices from tensor decompositions for this reduced model with lesser computational costs.
  \item We provide computational and memory complexity analysis for the CPD and TTD based methods in comparison to unfolding based matrix methods and demonstrate our framework in four numerical examples.
\end{enumerate}

The paper is organized into nine sections. In \cref{algebra}, we review tensor preliminaries including various notions of tensor products, tensor unfolding and properties of even-order paired tensors. \Cref{sec:3} introduces the MLTI system representation using the Einstein product and even-order paired tensors in detail. In \cref{sec:4}, we discuss notions of block tensors and tensor ranks/decompositions. We also build new results relating the unfolding rank of a tensor to other more standard notions of tensor ranks. We establish stability, reachability and observability conditions for MLTI systems in \cref{sec:5}. The application of generalized CPD/TTD for model reduction is discussed in \cref{sec:6}. Four numerical examples are presented in \cref{sec:7}. Finally, we summarize different numerical approaches associated with MLTI systems in \cref{sec:8} and conclude in \cref{sec:9} with future research directions.

\section{Tensor preliminaries}\label{algebra}
We take most of the concepts and notations for tensor algebra from the comprehensive works of Kolda et al. \cite{doi:10.1137/07070111X, Kolda06multilinearoperators} and Ragnarsson et al. \cite{doi:10.1137/110820609, RAGNARSSON2013853}. A tensor is a multidimensional array. The order of a tensor is the number of its dimensions, and each dimension is called a mode. An $N$-th order tensor usually is denoted by $\textsf{X}\in \mathbb{R}^{J_1\times J_2\times  \dots \times J_N}$. The sets of indexed indices and size of \textsf{X} are denoted by $\textbf{j}=\{j_1,j_2,\dots,j_N\}$ and $\mathcal{J}=\{J_1,J_2,\dots,J_N\}$, respectively. Let $\Pi_{\mathcal{J}}$ represent the product of all elements in $\mathcal{J}$, i.e. $\Pi_\mathcal{J} = \prod_{n=1}^NJ_n$. It is therefore reasonable to consider scalars $x\in\mathbb{R}$ as zero-order tensors, vectors $\textbf{v}\in\mathbb{R}^{J}$ as first-order tensors, and matrices $\textbf{A}\in\mathbb{R}^{J\times I}$ as second-order tensors.

\subsection{Tensor products}  By extending the notion of vector outer product, the outer product of two tensors $\textsf{X}\in \mathbb{R}^{J_1\times J_2\times \dots \times J_N}$ and $\textsf{Y}\in \mathbb{R}^{I_1\times I_2\times \dots \times I_M}$ is defined as
\begin{equation*}
(\textsf{X}\circ \textsf{Y})_{j_1j_2\dots j_Ni_1i_2\dots i_M}=\textsf{X}_{j_1j_2\dots j_N}\textsf{Y}_{i_1i_2\dots i_M}. 
\end{equation*}
In contrast, the inner product of two tensors $\textsf{X},\textsf{Y}\in \mathbb{R}^{J_1\times J_2\times \dots \times J_N}$ is defined as
\begin{equation*}
\langle \textsf{X},\textsf{Y}\rangle =\sum_{j_1=1}^{J_1}\dots \sum_{j_N=1}^{J_N}\textsf{X}_{j_1j_2\dots j_N}\textsf{Y}_{j_1j_2\dots j_N},
\end{equation*}
leading to the tensor Frobenius norm $\|\textsf{X}\|^2=\langle \textsf{X},\textsf{X}\rangle$. We say two tensors $\textsf{X}$ and $\textsf{Y}$ are orthogonal if the inner product $\langle \textsf{X},\textsf{Y}\rangle =0$. The matrix tensor multiplication $\textsf{X} \times_{n} \textbf{A}$ along mode $n$ for a matrix $\textbf{A}\in  \mathbb{R}^{I\times J_n}$ is defined by
$
(\textsf{X} \times_{n} \textbf{A})_{j_1j_2\dots j_{n-1}ij_{n+1}\dots j_N}=\sum_{j_n=1}^{J_n}\textsf{X}_{j_1j_2\dots j_n\dots j_N}\textbf{A}_{ij_n}.
$
This product can be generalized to what is known as the Tucker product, for $\textbf{A}_n\in \mathbb{R}^{I_n\times J_n}$,
\begin{equation}\label{eq5}
\begin{split}
\textsf{X}\times_1 \textbf{A}_1 \times_2\textbf{A}_2\times_3\dots \times_{N}\textbf{A}_N=\textsf{X}\times\{\textbf{A}_1,\textbf{A}_2,\dots,\textbf{A}_N\}\in  \mathbb{R}^{I_1\times I_2\times\dots \times I_N}.
\end{split}
\end{equation}

\subsection{Tensor unfolding}Tensor unfolding is considered as a critical operation in tensor computations \cite{doi:10.1137/07070111X, Kolda06multilinearoperators,doi:10.1137/110820609}. In order to unfold a tensor $\textsf{X}\in\mathbb{R}^{J_1\times J_2\times\dots \times J_N}$ into a vector or a matrix, we use an index mapping function $ivec(\cdot,\mathcal{J}):\mathbb{Z}^+\times \mathbb{Z}^+\times \stackrel{\scriptscriptstyle N}{\cdots} \times \mathbb{Z}^+ \rightarrow \mathbb{Z}^+$ as defined by Ragnarsson et al. \cite{doi:10.1137/110820609, RAGNARSSON2013853}, which is given by
\begin{equation*}
ivec(\textbf{j},\mathcal{J}) = j_1+\sum_{k=2}^N(j_k-1)\prod_{l=1}^{k-1}J_l.
\end{equation*}
The index mapping function $ivec$ returns the index for tensor vectorization, i.e. $\textbf{x}\in\mathbb{R}^{\Pi_{\mathcal{J}}}$ is the vectorization of \textsf{X} such that $\textbf{x}_{ivec(\textbf{j},\mathcal{J})}=\textsf{X}_{j_1j_2\dots j_N}$. If $N=2$,  $ivec$ will stack all the columns of \textsf{X}.

\begingroup
\setlength\arraycolsep{2pt}
For tensor matricization, let $z$ be an integer such that $1\leq z < N$, and $\mathbb{S}$ be a  permutation of $\{1,2,\dots,N\}$. If $\textbf{r} = \{\mathbb{S}(1),\mathbb{S}(2),\dots,\mathbb{S}(z)\}$ and $\textbf{c} = \{\mathbb{S}(z+1),\mathbb{S}(z+2),\dots,\mathbb{S}(N)\}$ with  $\mathcal{P} = \{J_{\mathbb{S}(1)},J_{\mathbb{S}(2)},\dots,J_{\mathbb{S}(z)}\}$ and $\mathcal{Q} = \{J_{\mathbb{S}(z+1)},J_{\mathbb{S}(z+2)},\dots,J_{\mathbb{S}(N)}\}$, respectively, the $\textbf{r}\textbf{c}$-unfolding matrix of $\textsf{X}$, denoted by $\textbf{X}_{(\textbf{rc})}\in\mathbb{R}^{\Pi_{\mathcal{P}}\times \Pi_{\mathcal{Q}}}$, is given by
\begin{equation}
(\textbf{X}_{(\textbf{rc})})_{pq} = \textsf{X}^{\mathbb{S}}_{p_1p_2\dots p_zq_1q_2\dots q_{N-z}},\label{eq20}
\end{equation}
where, $p = ivec(\textbf{p},\mathcal{P})$, $q = ivec(\textbf{q},\mathcal{Q})$, and $\textsf{X}^{\mathbb{S}}$ is the $\mathbb{S}$-transpose of $\textsf{X}$ defined as
 \begin{equation*}\textsf{X}^{\mathbb{S}}_{j_{\mathbb{S}(1)}j_{\mathbb{S}(2)}\dots j_{\mathbb{S}(N)}}=\textsf{X}_{j_1j_2\dots j_N}.
 \end{equation*}
When $z=1$ and $\mathbb{S} = \footnotesize{\begin{pmatrix} 1 & 2 & \dots& n & n+1 & \dots & N\\  n & 1 & \dots & n-1 & n+1 & \dots & N\end{pmatrix}}$, the tensor unfolding is called the $n$-mode matricization, denoted by $\textbf{X}_{(n)}$.
\endgroup

\subsection{Even-order paired tensors} Here we discuss the notion of even-order paired tensors and the Einstein product which will play an important role in developing the MLTI systems theory. 

\begin{definition}
Even-order paired tensors are $2N$-th order tensors with elements specified using a pairwise index notation, i.e. $\textsf{A}_{j_1i_1\dots j_Ni_N}$ for $\textsf{A}\in\mathbb{R}^{J_1\times I_1\times \dots \times J_N\times I_N}$.
\end{definition}

\begin{definition}
Given an even-order paired tensor $\textsf{A}\in\mathbb{R}^{J_1\times I_1\times \dots \times J_N\times I_N}$, the Einstein product between \textsf{A} and an $N$-th order tensor $\textsf{X}\in\mathbb{R}^{I_1\times I_2\times \dots \times I_N}$ is the contraction along the second index in each pair from \textsf{A}, i.e.
\begin{equation}\label{eq:0}
    (\textsf{A}* \textsf{X})_{j_1j_2\dots j_N} = \sum_{i_1=1}^{I_1}\dots\sum_{i_N=1}^{I_N} \textsf{A}_{j_1i_1\dots j_Ni_N}\textsf{X}_{i_1i_2\dots i_N}.
\end{equation}
\end{definition}

We use the pairwise index notation for even-order tensors because it is convenient for defining the unfolding transformation $\varphi$, see \cref{def:1}, and for representing core matrices/tensors in tensor decompositions, see \cref{sec:6.1}. Note that even-order paired tensors and the Einstein product (\ref{eq:0}) can be viewed as multidimensional generalizations of matrices and the standard matrix-vector product, respectively \cite{gel2017tensor}. Similar to the standard matrix-matrix product, one can also define a generalized form of the Einstein product between two even-order paired tensors. We will see later that the Einstein product can be efficiently computed using tensor decompositions of even-order paired tensors, see \cref{pro:2.13}.

\begin{definition}
Given two even-order paired tensors $\textsf{A}\in\mathbb{R}^{J_1\times K_1\times \dots J_N\times K_N}$ and $\textsf{B}\in\mathbb{R}^{K_1\times I_1 \times \dots \times K_N\times I_N}$, the Einstein product $\textsf{A}*\textsf{B}\in \mathbb{R}^{J_1\times I_1 \times \dots \times J_N\times I_N}$ is defined by
\begin{equation}\label{eq10}
(\textsf{A}* \textsf{B})_{j_1i_1\dots j_N i_N}=\sum_{k_1=1}^{K_1}\dots\sum_{k_N=1}^{K_N} \textsf{A}_{j_1k_1\dots j_Nk_N}\textsf{B}_{k_1i_1\dots k_Ni_N}.
\end{equation}
\end{definition}

Brazell et al. \cite{doi:10.1137/100804577} investigated properties of even-order tensors under the Einstein product (different from (\ref{eq10})) through construction of an isomorphism to GL($n,\mathbb{R}$), i.e. the set of $n\times n$ real valued invertible matrices. The existence of the isomorphism enables one to generalize several matrix concepts, such as invertibility and eigenvalue decomposition to the tensor case \cite{doi:10.1137/100804577, doi:10.1080/03081087.2015.1071311, He_2017, doi:10.1080/03081087.2018.1500993, doi:10.1080/03081087.2015.1083933}. We can establish an analogous isomorphism  for even-order paired tensors by a permutation of indices.
\begin{definition}\label{def:1}
Define the map $\varphi$: $\mathbb{T}_{J_1I_1\dots J_NI_N}(\mathbb{R}) \rightarrow \mathbb{M}_{\Pi_{\mathcal{J}}\Pi_{\mathcal{I}}}(\mathbb{R})$ with $\varphi(\textsf{A})=\textbf{A}$ defined component-wise as
\begin{equation}\label{eq:12}
\textsf{A}_{j_1i_1\dots j_Ni_N}\xrightarrow{\varphi} \textbf{A}_{ivec(\textbf{j},\mathcal{J})ivec(\textbf{i},\mathcal{I})},
\end{equation}
where, $\mathbb{T}_{J_1I_1\dots J_NI_N}(\mathbb{R})$ is the set of all real $J_1\times I_1\times \dots \times J_N\times I_N$ even-order paired tensors, and $\mathbb{M}_{\Pi_{\mathcal{J}}\Pi_{\mathcal{I}}}(\mathbb{R})$ is set of all real $\Pi_{\mathcal{J}}\times \Pi_{\mathcal{I}}$ matrices.
\end{definition}

\begingroup
\setlength\arraycolsep{2pt}
The map $\varphi$ can be viewed as a tensor unfolding discussed in (\ref{eq20}) with $z= N$ and $\mathbb{S} =\footnotesize{\begin{pmatrix} 1 & 2 & \dots & N  & N+1 & N+2 & \dots & 2N\\1 & 3 & \dots & 2N-1 & 2 & 4 & \dots & 2N\end{pmatrix}}$, so the Frobenius norm is preserved through $\varphi$, i.e. $\|\textsf{A}\| = \|\varphi(\textsf{A})\|$. More significantly, $\varphi$ is bijective, and the restriction of $\varphi^{-1}$ on the general linear group produces a group isomorphism.
\endgroup

\begin{corollary}
Let $J_n=I_n$ for all $n$ and $\mathbb{G}_{J_1J_1\dots J_NJ_N}(\mathbb{R})=\varphi^{-1}(\text{GL}(\Pi_{\mathcal{J}},\mathbb{R}))$, i.e. $\mathbb{G}_{J_1J_1\dots J_NJ_N}$ is the space of all even-order paired tensors which maps to the general linear group under $\varphi$. Then $\mathbb{G}_{J_1J_1\dots J_NJ_N}(\mathbb{R})$ is a group equipped with the Einstein product (\ref{eq10}), and $\varphi$ is a group isomorphism.
\end{corollary}

Detailed proofs can be found in \cite{doi:10.1137/100804577,He_2017}. Based on the unfolding property, we can define several tensor notations analogous to matrices as follows: 
\begin{enumerate}
  \item For an even-order paired tensor $\textsf{A} \in \mathbb{R}^{J_1\times I_1 \times \dots \times J_N\times I_N}$,  $\textsf{T} \in \mathbb{R}^{I_1\times J_1 \times \dots \times I_N\times J_N}$ is called the U-transpose of $\textsf{A}$ if $\textsf{T}_{i_1j_1\dots i_Nj_N} = \textsf{A}_{j_1i_1\dots j_Ni_N}$, and is denoted by $\textsf{A}^{\top}$. We refer to an even-order paired tensor that is identical to its U-transpose as weakly symmetric.
  \item For an even-order paired tensor $\textsf{A} \in \mathbb{R}^{J_1\times I_1 \times \dots \times J_N\times I_N}$, the unfolding rank of $\textsf{A}$ is defined as $\text{rank}_U(\textsf{A}) = \text{rank}(\varphi(\textsf{A}))$ \cite{doi:10.1080/03081087.2018.1500993}.
  \item An even-order ``square'' tensor $\textsf{D} \in \mathbb{R}^{J_1\times J_1 \times \dots \times J_N\times J_N}$ is called the U-diagonal tensor if all its entries are zeros except for $\textsf{D}_{j_1j_1\dots j_Nj_N}$. If all the diagonal entires $\textsf{D}_{j_1j_1\dots j_Nj_N}=1$, then $\textsf{D}$ is the U-identity tensor, denoted by $\textsf{I}$.
  \item  For an even-order square tensor $\textsf{A}\in \mathbb{R}^{J_1\times J_1 \times \dots \times J_N\times J_N}$, if there exists a tensor $\textsf{B}\in\mathbb{R}^{J_1\times J_1 \times \dots \times J_N\times J_N}$ such that $\textsf{A}*\textsf{B}=\textsf{B}*\textsf{A} = \textsf{I}$, then $\textsf{B}$ is called the U-inverse of $\textsf{A}$, denoted by $\textsf{A}^{-1}$.
      \item An even-order square tensor $\textsf{A}\in \mathbb{R}^{J_1\times J_1 \times \dots \times J_N\times J_N}$ is called U-positive definite if $\textsf{X}^{\top}*\textsf{A}*\textsf{X} > 0$ for any nonzero tensor $\textsf{X}\in \mathbb{R}^{J_1\times J_2\times\dots \times J_N}$.
  \item For an even-order square tensor $\textsf{A}\in \mathbb{R}^{J_1\times J_1 \times \dots \times J_N\times J_N}$, the unfolding determinant of $\textsf{A}$ is defined as $\text{det}_U(\textsf{A}) = \text{det}(\varphi(\textsf{A}))$ \cite{doi:10.1080/03081087.2018.1500993}.
\end{enumerate}
In \cref{app:1}, we show that the notion of U-positive definiteness is a generalization of M-positive definiteness and rank-one positive definiteness proposed in \cite{Huang_2017, Liqun2009ConditionsFS} for the even-order elasticity tensors.

\section{MLTI system representation}\label{sec:3}
To describe the evolution of tensor time series, the authors in \cite{rogers_2013,7798500} introduced a MLTI system using the Tucker product, which can be generalized by incorporating control inputs as follows:
\begin{align}
\begin{cases}
\textsf{X}_{t+1}=\textsf{X}_{t}\times\{\textbf{A}_1,\dots,\textbf{A}_N\}+\textsf{U}_t\times\{\textbf{B}_1,\dots,\textbf{B}_N\}\\
\textsf{Y}_{t}=\textsf{X}_{t}\times\{\textbf{C}_1,\dots,\textbf{C}_N\}
\end{cases},\label{eq9}
\end{align}
where, $\textsf{X}_{t}\in\mathbb{R}^{J_1\times J_2\times \dots \times J_N}$ is the latent state space tensor, $\textsf{Y}_{t}\in\mathbb{R}^{I_1\times I_2\times \dots \times I_N}$ is the output tensor, and $\textsf{U}_t\in \mathbb{R}^{K_1\times K_2\times \dots \times K_N}$ is a control tensor. $\textbf{A}_n\in \mathbb{R}^{J_n\times J_n}$, $\textbf{B}_n\in\mathbb{R}^{J_n\times K_n}$ and $\textbf{C}_n\in \mathbb{R}^{I_n\times J_n}$ are real valued matrices for $n=1,2,\dots,N$. The Tucker product provides a suitable way to deal with MLTI systems because it allows one to exploit matrix computations. However, we find that (\ref{eq9}) can be replaced by a more general representation using the notion of even-order paired tensors and the Einstein product. Moreover, the new representation is more concise and systematic compared to the tensor based linear system proposed in \cite{DING201875}.
\begin{definition}
A more general representation of MLTI system is given by
\begin{align}
\begin{cases}
\textsf{X}_{t+1}=\textsf{A}*\textsf{X}_{t}+\textsf{B}*\textsf{U}_{t}\\
\textsf{Y}_{t}=\textsf{C}*\textsf{X}_{t}
\end{cases},\label{eq11}
\end{align}
where, $\textsf{A}\in\mathbb{R}^{J_1\times J_1\times \dots \times J_N\times J_N}$, $\textsf{B}\in\mathbb{R}^{J_1\times K_1\times \dots \times J_N\times K_N}$ and $\textsf{C}\in\mathbb{R}^{I_1\times J_1\times \dots \times I_N\times J_N}$ are even-order paired tensors.
\end{definition}

\begin{lemma}\label{lem:2.9}
Let $\textsf{A}\in \mathbb{R}^{J_1\times I_1 \times \dots \times J_N\times I_N}$ be an even-order paired tensor. Then the product $\textsf{A}\times \{\textbf{U}_1,\textbf{V}_1,\dots, \textbf{U}_N,\textbf{V}_N\} = \textsf{U}*\textsf{A}*\textsf{V}^\top\in\mathbb{R}^{K_1\times L_1\times \dots \times K_N\times L_N}$ for $\textsf{U} = \textbf{U}_1\circ\textbf{U}_2\circ\dots \circ \textbf{U}_N$ and $\textsf{V} = \textbf{V}_1\circ\textbf{V}_2\circ\dots \circ \textbf{V}_N$ where $\textbf{U}_n\in\mathbb{R}^{K_n\times J_n}$ and  $\textbf{V}_n\in\mathbb{R}^{L_n\times I_n}$.
\end{lemma}
\begin{proof}
This follows from the definitions of the Tucker and Einstein products.
\end{proof}

\begin{proposition}\label{pro:3.2}
The governing equations (\ref{eq11}) can be obtained from (\ref{eq9}) by setting $\textsf{A}$, $\textsf{B}$ and $\textsf{C}$ to be the outer products of component matrices $\{\textbf{A}_1,\textbf{A}_2,\dots,\textbf{A}_N\}$, $\{\textbf{B}_1,\textbf{B}_2,\dots,\textbf{B}_N\}$ and $\{\textbf{C}_1,\textbf{C}_2,\dots,\textbf{C}_N\}$, respectively.
\end{proposition}
\begin{proof}
The result follows from \cref{lem:2.9} with $I_n=1$ and $\textbf{V}_n=1$ for all $n$.
\end{proof}

The main advantages of the MLTI system (\ref{eq11}) are as follows: 
\begin{enumerate}
\item The Einstein product representation (\ref{eq11}) of MLTI systems is indeed the generalization of (\ref{eq9}). While \cref{pro:3.2} shows that MLTI systems in the form of (\ref{eq9}) can always be transformed into the form of (\ref{eq11}), the converse is not always true, see  (\ref{eq45}) for example. It is true only when $R_1=R_2=R_3=1$.
\item  The MLTI system (\ref{eq11}) takes a form  similar to the standard LTI system model with matrix product replaced with the Einstein product, so the representation is more natural for developing the MLTI systems theory including notions of stability, reachability and observability. Moreover, the concept of transfer functions which is commonly used in modern control theory can be extended for MLTI systems, see \cref{def:3.4}. 
\item We can exploit tensor decompositions (see \cref{sec:4.2}) of the even-order paired tensors \textsf{A}, \textsf{B} and \textsf{C} to accelerate computations in MLTI systems theory. In particular, if \textsf{A}, \textsf{B} and \textsf{C} possess low tensor rank structures, we can obtain a low-parameter MLTI representation. In addition, many operations such as the Einstein product and unfolding rank can be achieved efficiently in the tensor decomposition format compared to unfolding based matrix methods, see remarks in \cref{sec:5,sec:6}.
\item Traditional LTI model reduction and identification techniques such as balanced truncation and eigensystem realization algorithm can be extended using the form of  (\ref{eq11}). 
\end{enumerate}

\begin{definition}\label{def:3.4}
The transfer function $\textsf{G}(z)$ of (\ref{eq11}) is given by
\begin{equation}\label{eq:3.20}
\textsf{G}(z) = \textsf{C}*(z\textsf{I}-\textsf{A})^{-1}*\textsf{B},
\end{equation}
where, $z$ is a complex variable.
\end{definition}

We first investigate the elementary solution to the MLTI system (\ref{eq11}), which is crucial in the analysis of stability, reachability and observability.
\begin{proposition}\label{pro9}
For an unforced MLTI system
$
\textsf{X}_{t+1}=\textsf{A}*\textsf{X}_t,
$
the solution for $\textsf{X}$ at time $k$, given initial condition $\textsf{X}_0$, is $\textsf{X}_k=\textsf{A}^{k}* \textsf{X}_0$ where $\textsf{A}^{k}=\textsf{A}*\textsf{A}*\stackrel{ k}{\cdots}*\textsf{A}$.
\end{proposition}

The proof is straightforward using the notion of even-order paired tensors and the Einstein product. Applying \cref{pro9}, we can write down the explicit solution of (\ref{eq11}) which takes an analogous form to the LTI system
\begin{equation}\label{eq:31}
\textsf{X}_k=\textsf{A}^{k}* \textsf{X}_0+\sum_{j=0}^{k-1}\textsf{A}^{k-j-1}*\textsf{B}*\textsf{U}_j.
\end{equation}

Lastly, we want to note that one can always transform the MLTI system (\ref{eq11})  into a LTI system using $\varphi$, i.e. $\textbf{x}_{t+1} = \varphi(\textsf{A})\textbf{x}_t+\varphi(\textsf{B})\textbf{u}_t$, and determine the stability, reachability and observability using classical matrix techniques.

\section{Tensor algebra continued} \label{sec:4}
We next discuss notions of block tensors and tensor decompositions which will form the basis for developing tensor algebra based concepts of stability, reachability and observability of the MLTI system (\ref{eq11}).
\subsection{Block tensors}\label{sec:4.1}
Analogously to  block matrices, one can define the notion of block tensors. For tensors of the same size, we propose a block tensor construction (first appeared in \cite{Chen_2018}) which does not introduce any wasteful zeros compared to the block tensors proposed in \cite{doi:10.1080/03081087.2015.1083933}, and thus offers computational advantages.

\begin{definition}\label{def:nmodedef}
Let  $\textsf{A},\textsf{B}\in\mathbb{R}^{J_1\times I_1\times \dots \times J_N\times I_N}$ be two even-order paired tensors. The $n$-mode row block tensor $\begin{vmatrix}
	\textsf{A} & \textsf{B}
\end{vmatrix}_n\in \mathbb{R}^{J_1\times I_1 \times \dots \times J_n\times 2I_n \times\dots \times J_N\times I_N}$ is defined by
\begin{equation*}
(\begin{vmatrix}
	\textsf{A} & \textsf{B}
\end{vmatrix}_n)_{j_1l_1\dots j_Nl_N} = 
\begin{cases}
\textsf{A}_{j_1l_1\dots j_Nl_N},&\text{ } j_k=1,\dots,J_k, \text{ } l_k=1,\dots,I_k \text{ } \forall k\\
\textsf{B}_{j_1l_1\dots j_Nl_N}, &\text{ }j_k=1,\dots,J_k \text{ }\forall k, l_k = 1,\dots, I_k \text{ for } k\neq n\\
&\text{ and } l_k=I_k+1,\dots,2I_k \text{ for } k=n
\end{cases}.
\end{equation*}
\end{definition}

For example, if $\textsf{A},\textsf{B}\in\mathbb{R}^{2\times 2\times 2\times 2}$, then the 1-mode row block tensor is given by $\begin{vmatrix}
	\textsf{A} & \textsf{B}
\end{vmatrix}_1\in \mathbb{R}^{2\times 4\times 2\times 2}$ such that $(\begin{vmatrix}
	\textsf{A} & \textsf{B}
\end{vmatrix}_1)_{:i_1::} = \textsf{A}$ for $i_1=1,2$ and $(\begin{vmatrix}
	\textsf{A} & \textsf{B}
\end{vmatrix}_1)_{:i_1::} = \textsf{B}$ for $i_1=3,4$. Similarly for $\begin{vmatrix}
	\textsf{A} & \textsf{B}
\end{vmatrix}_2\in \mathbb{R}^{2\times 2\times 2\times 4}$. Detailed explanations of the MATLAB colon operation \texttt{:} can be found in \cref{app:5.1}. When $N=1$, it reduces to the row block matrices.  The $n$-mode column block tensor \begin{equation*}
\begin{vmatrix}
	\textsf{A}\\
	\textsf{B}
\end{vmatrix}_n\in \mathbb{R}^{J_1\times I_1 \times \dots \times 2J_n\times I_n \times\dots \times J_N\times I_N}\end{equation*}
can be defined in a similar manner. The $n$-mode block tensors exhibit many  properties analogous to  block matrix computations, e.g. the Einstein product can distribute over block tensors, and  the blocks of $n$-mode row block tensors map to contiguous blocks under $\varphi$ up to some permutations \cite{doi:10.1137/110820609}, see details in \cref{app:2}. Therefore, rank is preserved in the block tensor unfolding, i.e.
$
\text{rank}_U(\begin{vmatrix}
	\textsf{A} & \textsf{B}
\end{vmatrix}_n) = \text{rank}(\begin{bmatrix}
	\varphi(\textsf{A}) & \varphi(\textsf{B})
\end{bmatrix}),
$ where $[\cdot]$ denotes the block matrix operation.

Given $K$ even-order paired tensors $\textsf{X}_n\in\mathbb{R}^{J_1\times I_1\times \dots\times J_N\times I_N}$, one can apply \cref{def:nmodedef} successively to create a $J_1\times I_1\times \dots \times J_n\times I_nK  \times \dots \times J_N\times I_N$ even-order $n$-mode row block tensor. However, a more general concatenation approach can be defined for multiple blocks.

\begin{definition}
Given $K$ even-order paired tensors $\textsf{X}_n\in\mathbb{R}^{J_1\times I_1\times \dots\times J_N\times I_N}$, if $K=K_1K_2\dots K_N$, the $J_1\times I_1K_1\times \dots \times J_N\times I_NK_N$ even-order mode row block tensor $\textsf{Y}$ can be constructed in the following way:
\begin{enumerate}
 \item Compute the 1-mode row block tensor concatenation over  $\{\textsf{X}_1,\cdots,\textsf{X}_{K_1}\}$, $\{\textsf{X}_{K_1+1},\cdots,\textsf{X}_{2K_1}\}$ and so on to obtain $K_2K_3\dots K_N$ block tensors denoted by $\textsf{X}_1^{(1)},\textsf{X}_2^{(1)},\dots,\textsf{X}_{K_2K_3\dots K_N}^{(1)}$;
 \item Compute the 2-mode row block tensors concatenation over $\{\textsf{X}_1^{(1)},\cdots,\textsf{X}_{K_2}^{(1)}\}$, $\{\textsf{X}_{K_2+1}^{(1)},\cdots,\textsf{X}_{2K_2}^{(1)}\}$ and so on to obtain $K_3K_4\dots K_N$ block tensors denoted by $\textsf{X}_1^{(2)},\textsf{X}_2^{(2)},\dots,  \textsf{X}_{K_3K_4\dots K_N}^{(2)}$;
 \item Keep repeating the process until the last $N$-mode row block tensor is obtained.
 \end{enumerate}
  We denote the mode row block tensor as $\textsf{Y} = \begin{vmatrix} \textsf{X}_1 &\textsf{X}_2 & \dots & \textsf{X}_{K}\end{vmatrix}$.
\end{definition}

For example, suppose that $\textsf{X}_n\in\mathbb{R}^{2\times 2\times 2\times 2\times 2\times 2}$ for $n=1,2,\dots,K$ and $K=8$. Let  $K=K_1K_2K_3$ with $K_1=K_2=K_3=2$. Given this factorization of $K$, the mode row block tensor $\textsf{Y}\in\mathbb{R}^{2\times 4\times 2\times 4\times 2\times 4}$ is constructed in the manner shown in \cref{fig2}, in which $\textsf{X}^{(1)}_n\in\mathbb{R}^{2\times 4\times 2\times 2\times 2\times 2}$ and $\textsf{X}^{(2)}_n\in\mathbb{R}^{2\times 4\times 2\times 4\times 2\times 2}$. Another factorization with $K_1=2$, $K_2=4$ and $K_3=1$ would return $\textsf{Y}\in\mathbb{R}^{2\times 4\times 2\times 8\times 2\times 2}$. The generalized mode column block tensors with multiple blocks can be constructed in a similar manner. When $I_n=1$ for all $n$, the above generalized mode row block tensor maps exactly to contiguous blocks in its unfolding under $\varphi$, which could be beneficial in many block tensor applications. Furthermore, the choices of $K_n$  may affect the structure of  mode block tensors, which can be significant in tensor ranks/decompositions \cite{chen_2019}. 

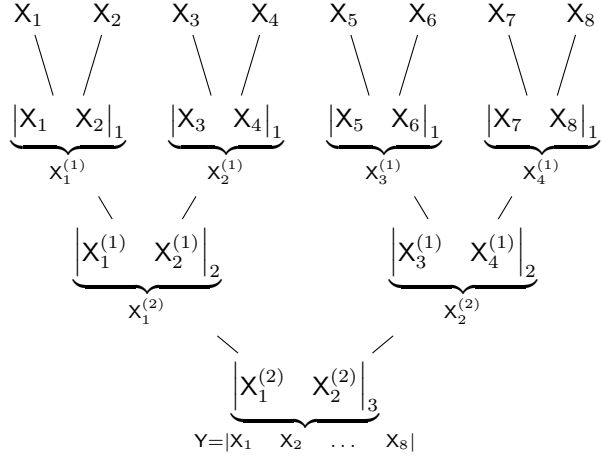
\begin{figure}[h]\label{fig2}
\centering
\begin{tikzpicture}[scale = 0.35,level distance=5cm]
\tikzstyle{level 1}=[sibling distance=12cm]
\tikzstyle{level 2}=[sibling distance=6cm]
\tikzstyle{level 3}=[sibling distance= 3cm]
\node {$\underbrace{\begin{vmatrix} \textsf{X}^{(2)}_1 &\textsf{X}^{(2)}_2\end{vmatrix}_3}_{\scriptsize{\textsf{Y}=\begin{vmatrix} \textsf{X}_1 &\textsf{X}_2 & \dots & \textsf{X}_{8}\end{vmatrix}}}$} [grow'=up]
child {node {$\underbrace{\begin{vmatrix} \textsf{X}^{(1)}_1 &\textsf{X}^{(1)}_2\end{vmatrix}_2}_{\textsf{X}_1^{(2)}}$}
child {node {$\underbrace{\begin{vmatrix} \textsf{X}_1 &\textsf{X}_2\end{vmatrix}_1}_{\textsf{X}_1^{(1)}}$}
child {node {$\textsf{X}_1$}}
child {node {$\textsf{X}_2$}}}
child {node {$\underbrace{\begin{vmatrix} \textsf{X}_3 &\textsf{X}_4\end{vmatrix}_1}_{\textsf{X}_2^{(1)}}$}
child {node {$\textsf{X}_3$}}
child {node {$\textsf{X}_4$}}}
}
child {node {$\underbrace{\begin{vmatrix} \textsf{X}^{(1)}_3 &\textsf{X}^{(1)}_4\end{vmatrix}_2}_{\textsf{X}_2^{(2)}}$}
child {node {$\underbrace{\begin{vmatrix} \textsf{X}_5 &\textsf{X}_6\end{vmatrix}_1}_{\textsf{X}_3^{(1)}}$}
child {node {$\textsf{X}_5$}}
child {node {$\textsf{X}_6$}}}
child {node {$\underbrace{\begin{vmatrix} \textsf{X}_7 &\textsf{X}_8\end{vmatrix}_1}_{\textsf{X}_4^{(1)}}$}
child {node {$\textsf{X}_7$}}
child {node {$\textsf{X}_8$}}}
};
\end{tikzpicture}
\caption{An example of mode row block tensor.}
\end{figure}

\subsection{Tensor ranks and decompositions}\label{sec:4.2}
There are several definitions of tensor ranks \cite{doi:10.1137/06066518X, doi:10.1137/07070111X, Kolda06multilinearoperators}, which are intimately related to different notions of tensor decompositions. The multilinear ranks or the $n$-ranks of $\textsf{X}$ are the ranks of the $n$-mode matricizations, denoted by $\text{rank}_n(\textsf{X})$. The multilinear ranks are related to the so-called Higher-Order Singular Value Decomposition (HOSVD), a multilinear generalization of the matrix Singular Value Decomposition (SVD) \cite{5447070, doi:10.1137/S0895479896305696}.
\begin{theorem}[HOSVD] \label{thm:2.1}
A tensor $\textsf{X}\in\mathbb{R}^{J_1\times J_2\times \dots \times J_N}$ can be written as
\begin{equation}
\textsf{X} = \textsf{S}\times_1 \textbf{U}_1\times_2\dots \times_N \textbf{U}_N,
\end{equation}
where, $\textbf{U}_n\in\mathbb{R}^{J_n\times J_n}$ are orthogonal matrices, and $\textsf{S}\in\mathbb{R}^{J_1\times J_2\times \dots \times J_N}$ is a tensor of which the subtensors $\textsf{S}_{j_n=\alpha}$ obtained by fixing the $n$-th index to $\alpha$, have the properties:
\begin{enumerate}
\item all-orthogonality: two subtensors $\textsf{S}_{j_n=\alpha}$ and $\textsf{S}_{j_n=\beta}$ are orthogonal for all possible values of $n$, $\alpha$ and $\beta$ subject to $\alpha\neq \beta$;
\item ordering: $\|\textsf{S}_{j_n=1}\|\geq \dots \geq \|\textsf{S}_{j_n=J_n}\|\geq 0$ for all possible values of $n$.
\end{enumerate}
The Frobenius norms $\|\textsf{S}_{j_n=j}\|$, denoted by $\gamma_{j}^{(n)}$, are the $n$-mode singular values of $\textsf{X}$.
\end{theorem}

De Lathauwer et al. \cite{doi:10.1137/S0895479896305696} showed that the number of nonvanishing $n$-mode singular values from the HOSVD of a tensor is equal to its $n$-mode multilinear rank. In addition, the error bound of the low mutilinear rank approximation is provided in \cite{doi:10.1137/S0895479896305696}. Unlike the matrix SVD, the approximation fails to obtain the best rank approximation of $\textsf{X}$. Nevertheless, it still can provide a ``good'' estimate with appropriate $n$-mode singular values truncated \cite{doi:10.1137/S0895479896305696}.

Analogous to rank-one matrices, a tensor $\textsf{X}$ is rank-one if it can be written as the outer product of $N$ vectors, i.e.
$
\textsf{X}=\textbf{a}^{(1)}\circ  \textbf{a}^{(2)} \circ \dots \circ \textbf{a}^{(N)}.
$
The CANDECOMP/PARAFAC Decomposition (CPD) decomposes a tensor $\textsf{X}\in\mathbb{R}^{J_1\times J_2\times \dots \times J_N}$ into a sum of rank-one tensors as form of outer products. It is often useful to normalize all the vectors and have weights $\lambda_r>0$ in descending order in front:
\begin{equation}\label{eq:8}
\textsf{X}=\sum_{r=1}^{R}\lambda_r\textbf{a}_r^{(1)}\circ  \textbf{a}_r^{(2)} \circ \dots \circ \textbf{a}_r^{(N)},
\end{equation}
where, $\textbf{a}_r^{(n)}\in\mathbb{R}^{J_n}$ have unit length, and $R$ is called the CP rank of $\textsf{X}$ if it is the minimum integer that achieves (\ref{eq:8}). The factor matrices $\textbf{A}^{(n)}\in\mathbb{R}^{J_n\times R}$ are the combination of the vectors from the rank-one components for $n=1,2,\dots, N$, i.e.  $\textbf{A}^{(n)}=\begin{bmatrix} \textbf{a}_1^{(n)} & \textbf{a}_2^{(n)} & \dots & \textbf{a}_R^{(n)}\end{bmatrix}$. The CPD is unique up to scaling and permutation under a weak condition: for $N\geq 2$ and $R\geq 2$,
$
\sum_{n=1}^N k_{\textbf{A}^{(n)}}\geq 2R+(N-1),
$
where $k_{\textbf{A}^{(n)}}$, called the $k$-rank of a matrix, is the maximum number of columns of $\textbf{A}^{(n)}$ that are linearly independent with each other \cite{KRUSKAL197795, doi:10.1002/1099-128X, STEGEMAN2007540}.

The CP rank of a tensor is always greater than or equal to its multilinear ranks \cite{doi:10.1137/06066518X}. In fact, it is greater than or equal to any unfolding matrix rank \cite{doi:10.1137/090748330} (which can be used in unfolding rank and TT-ranks defined later too). The best CP rank approximation is ill-posed \cite{doi:10.1137/06066518X}, but carefully truncating the CP rank will yield a good estimate of the original tensor. Both CPD and HOSVD are special cases of Tucker Decomposition, which decomposes a tensor into the form of Tucker product (\ref{eq5}), i.e. $\textsf{Y} = \textsf{X}\times\{\textbf{A}_1,\textbf{A}_2,\dots,\textbf{A}_N\}$ \cite{KRUPPA20175610}.

The Tensor Train Decomposition (TTD) of $\textsf{X}\in\mathbb{R}^{J_1\times J_2\times \dots \times J_N}$ is given by
\begin{equation}\label{train}
\textsf{X} = \sum_{r_0=1}^{R_0}\dots \sum_{r_N=1}^{R_N}\textsf{X}^{(1)}_{r_0:r_1} \circ \textsf{X}^{(2)}_{r_1:r_2}\circ \dots \circ \textsf{X}^{(N)}_{r_{N-1}:r_N},
\end{equation}
where, $\{R_0,R_1,\dots,R_N\}$ is the set of TT-ranks with $R_0=R_N=1$, and $\textsf{X}^{(n)}\in\mathbb{R}^{R_{n-1}\times J_n\times R_{n}}$ are called the core tensors of the TTD \cite{doi:10.1137/090752286}. Here we have used : for brevity of notation, for the full definition see \cref{app:5.1}. Standard TTD algorithms, such as Algorithm 3 in \cite{Klus_2018}, with zero truncation will return the optimal TT-ranks 
\begin{equation*}
R_n=\text{rank}(\texttt{reshape}(\textsf{X},\prod_{i=1}^nJ_i,\prod_{i=n+1}^N J_i))
\end{equation*}
 for $n=1,2,\dots,N-1$. A core tensor $\textsf{X}^{(n)}$ is left-orthonormal if $(\bar{\textbf{X}}^{(n)})^\top\bar{\textbf{X}}^{(n)} = \textbf{I}\in\mathbb{R}^{R_n\times R_n}$, and is right-orthonormal if $\underline{\textbf{X}}^{(n)}(\underline{\textbf{X}}^{(n)})^\top = \textbf{I}\in\mathbb{R}^{R_{n-1}\times R_{n-1}}$, where $\bar{\textbf{X}}^{(n)}= \texttt{reshape}(\textsf{X}^{(n)},R_{n-1}J_n,R_n)$, and $\underline{\textbf{X}}^{(n)} = \texttt{reshape}(\textsf{X}^{(n)},R_{n-1},J_nR_n)$, respectively \cite{DOLGOV20141207,Klus_2018}. Here $\texttt{reshape}$ refers to the reshape operation in MATLAB, see details in \cref{app:5.2}. Detailed algorithms for left- and right-orthonormalization can be found in \cite{Klus_2018}. TTD is advantageous in that it provides better compression, i.e. truncating the TT-ranks results in a quasi-optimal approximation of $\textsf{X}$, and is computationally more robust \cite{doi:10.1137/090752286}.

 Eigenvalue problems for tensors were first explored by Qi \cite{QI20051302} and Lim \cite{singularvaluetensor} independently. Brazell et al. \cite{doi:10.1137/100804577} formulated a new tensor eigenvalue problem through the isomorphism $\varphi$ for fourth-order tensors, and Cui et al. \cite{doi:10.1080/03081087.2015.1071311} extended the tensor eigenvalue problem to even-order tensors.  
\begin{definition}\label{def:2.10}
Let $\textsf{A}\in\mathbb{R}^{J_1\times J_1\times \dots \times J_N\times J_N}$ be an even-order square tensor. If $\textsf{X}\in\mathbb{C}^{J_1\times J_2\dots\times J_N}$ is a nonzero $N$-th order tensor, $\lambda\in \mathbb{C}$, and $\textsf{X}$ and $\lambda$ satisfy
$
\textsf{A}*\textsf{X}=\lambda\textsf{X},
$
then we call $\lambda$ and $\textsf{X}$ as the U-eigenvalue and U-eigentensor of $\textsf{A}$, respectively.
\end{definition}

The algebraic and geometric multiplicities of U-eigenvalues can be defined as for matrices. The generalization of the Caley-Hamilton theorem for the tensor case can be obtained by the isomorphism property, i.e. an even-order square tensor $\textsf{A}$ satisfies its own characteristic polynomial $p(\lambda) = \text{det}_U(\lambda \textsf{I}-\textsf{A})$. Moreover, it can be shown that the notion of U-eigenvalues is a generalization of Z-eigenvalues and M-eigenvalues as proposed in \cite{Huang_2017, singularvaluetensor, QI20051302}. Detailed proofs are omitted in this paper. 

\begin{proposition}\label{prop:TEVD}
The tensor eigenvalue problem in \cref{def:2.10} can be represented by 
$
\textsf{A} = \textsf{V} *\textsf{D}* \textsf{V}^{-1} \label{eq216}
$
where $\textsf{D}\in\mathbb{R}^{J_1\times J_1\times \dots \times J_N\times J_N}$ is an U-diagonal tensor with U-eigenvalues on its diagonal, and $\textsf{V}\in\mathbb{R}^{J_1\times J_1\times \dots \times J_N\times J_N}$ is a mode row block tensor consisting of all the U-eigentensors, i.e.
$
\textsf{V} = \begin{vmatrix} \textsf{X}_1 & \textsf{X}_2&\dots & \textsf{X}_{\Pi_{\mathcal{J}}}\end{vmatrix}.
$
We have chosen $K_n=J_n$ in applying the mode row block tensor operation which enables to express the TEVD in the form analogous to the matrix case.
\end{proposition}
\begin{proof}
The proof follows immediately from \cref{pro:2.9}. 
\end{proof}

\subsection{Rank relations}\label{sec:4.3}
We establish new results relating the unfolding rank of an even-order paired tensor to its multilinear ranks, CP rank and TT-ranks. These relationships are useful for checking multilinear generalizations of reachability and observability rank conditions.

\begin{proposition}\label{pro:2.95}
Let $\textsf{A}\in\mathbb{R}^{J_1\times I_1\times \dots\times J_N\times I_N}$ be an even-order paired tensor. If $\text{rank}_U(\textsf{A}) = \Pi_{\mathcal{J}}$ (or $\text{rank}_U(\textsf{A}) = \Pi_{\mathcal{I}}$), then $\text{rank}_{2n-1} (\textsf{A}) = J_n$ (or $\text{rank}_{2n} (\textsf{A}) = I_n$) for $n=1,2,\dots,N$.
\end{proposition}

\begin{proposition}\label{pro:2.11}
Let $\textsf{A}\in\mathbb{R}^{J_1\times I_1\times \dots\times J_N\times I_N}$ be an even-order paired tensor given in the CPD format (\ref{eq:8}) with CP rank equal to $R$. If the following conditions
\begin{equation}
\sum_{n=1:2}^{2N} k_{\textbf{A}^{(n)}} \geq R + N -1 \text{, }
\sum_{n=2:2}^{2N} k_{\textbf{A}^{(n)}} \geq R + N -1
\label{eq:2.17}
\end{equation}
are satisfied for every $k_{\textbf{A}^{(n)}}\geq 1$, then $\text{rank}_U(\textsf{A})= R$.
\end{proposition}

The notations $\sum_{n=1:2}^{2N}$ and $\sum_{n=2:2}^{2N}$ represent the sums of all odd indices and all even indices, respectively. The detailed proofs of \cref{pro:2.95,pro:2.11} can be found in \cref{proof:2.95,proof:2.11}, respectively.

\begingroup
\setlength\arraycolsep{2pt}
\begin{proposition}\label{pro:2.96}
Let $\textsf{A}\in\mathbb{R}^{J_1\times I_1\times \dots\times J_N\times I_N}$ be an even-order paired tensor. Then $\text{rank}_U(\textsf{A}) = \tilde{R}_N$ where $\tilde{R}_N$ is the $N$-th optimal TT-rank of $\tilde{\textsf{A}}$, the $\mathbb{S}$-transpose of \textsf{A} with $\mathbb{S} =\footnotesize{\begin{pmatrix} 1 & 2 & \dots & N  & N+1 & N+2 & \dots & 2N\\1 & 3 & \dots & 2N-1 & 2 & 4 & \dots & 2N\end{pmatrix}}$.
\end{proposition}
\begin{proof}
The result follows from the definition of optimal TT-ranks.
\end{proof}
\endgroup

\textit{Remark.}
Given the TTD of \textsf{A}, the  TTD of $\tilde{\textsf{A}}\in\mathbb{R}^{J_1\times \dots \times J_N\times I_1\times \dots \times I_N}$ can be constructed by manipulating the core tensors $\textsf{A}^{(n)}$ without converting back to the full format. Assume that $J_n=I_n=J$ for all $n$, and $R$ is the average of the TT-ranks of \textsf{A}. If $R$ remains unchanged or decreases during this conversion, the computational complexity is estimated to be at most $\mathcal{O}(N^2J^3R^3)$\footnote{Big O notation: $f(x)=\mathcal{O}(g(x))$ as $x\rightarrow \infty$ if and only if there exists a positive real number $M$ and a real number $x_0$ such that $|f(x)|\leq Mg(x)$ for all $x\geq x_0$.}.  A detailed algorithm for the TTD based permutation can be found in \cite{chen_2019}.

\section{MLTI systems theory}\label{sec:5}
We now introduce the concepts of stability, reachability and observability for MLTI systems. Note that some preliminary results have appeared in our introductory paper \cite{Chen_2018}.
\subsection{Stability}\label{sec:5.1}
There are many notions of stability for dynamical systems \cite{doi:10.1137/1.9781611973884, Kailath_1980, Rugh:1996:LST:225486}. For LTI systems, it is conventional to investigate so-called internal stability. Generalizing from LTI systems, the equilibrium point $\textsf{X}=\textsf{O}$ (\textsf{O} denotes the zero tensors) of an unforced MLTI system is called stable if $\|\textsf{X}_t\|\leq \gamma \|\textsf{X}_0\|$ for some $\gamma>0$, asymptotically stable if $\|\textsf{X}_t\| \rightarrow 0$ as $t\rightarrow \infty$, and unstable if it is not stable. %\cite{Scruggs_2017}.

\begin{proposition}\label{pro:3.5}
Let $\lambda_j$ be the U-eigenvalues of \textsf{A} for $j=ivec(\textbf{j},\mathcal{J})$. For an unforced MLTI system, the equilibrium point $\textsf{X}=\textsf{O}$ is:
\begin{enumerate}
\item stable if and only if $|\lambda_j|\leq 1$ for all $j=1,2,\dots,\Pi_{\mathcal{J}}$; for those equal to 1, its algebraic and geometry multiplicities must be equal;
\item asymptotically stable if $|\lambda_j|<1$ for all $j=1,2,\dots,\Pi_{\mathcal{J}}$;
\item unstable if $|\lambda_j|>1$ for some $j=1,2,\dots,\Pi_{\mathcal{J}}$.
\end{enumerate}
\end{proposition}
\begin{proof}
%The definitions of algebraic and geometry multiplicity are defined analogously as in vector spaces.
We only focus on the case when  $\textsf{A}$ has a full set of U-eigentensors. It follows from \cref{prop:TEVD,pro9} that
$
\textsf{A}^{k} = \sum_{j_1=1}^{J_1}\dots\sum_{j_N=1}^{J_N} \lambda^k_j \textsf{W}_{j_1j_1\dots j_Nj_N}
$
for some even-order square tensors $\textsf{W}_{j_1j_1\dots j_Nj_N}$. Then the results follow immediately.
\end{proof}

\begin{corollary}\label{cor:5.2}
Suppose that the HOSVD of $\textsf{A}$ is provided with $n$-mode singular values. For an unforced MLTI system, the equilibrium point $\textsf{X}=\textsf{O}$ is asymptotically stable if the sum of the $n$-mode singular values square is less than one for any $n$.
\end{corollary}
\begin{proof}
Without loss of generality, suppose that $n=1$. Based on Property 8 in \cite{doi:10.1137/S0895479896305696}, 
$
\sum_{j=1}^{J_1} (\gamma_{j}^{(1)})^2=\|\textsf{A}\|^2=\|\varphi(\textsf{A})\|^2.
$
In addition, we know that the magnitude of the maximal eigenvalue of a matrix is less than or equal to its Frobenius norm. Hence, the proof follows immediately from \cref{pro:3.5}.
\end{proof}

\begin{corollary}\label{cor:5.3}
Suppose that the CPD of \textsf{A} is provided and its factor matrices $\textbf{A}^{(n)}$ and $\textbf{A}^{(m)}$ have all the column vectors orthonormal for at least one odd $n$ and even $m$. For an unforced MLTI system, the equilibrium point $\textsf{X}=\textsf{O}$ is asymptotically stable if the first weight element $\lambda_1<1$. 
\end{corollary}

The proof of \cref{cor:5.3} is presented in \cref{app:2.3}.

\begin{corollary}\label{coro:3.7}
Suppose that the TTD of $\tilde{\textsf{A}}\in\mathbb{R}^{J_1\times \dots \times J_N\times J_1\times \dots \times J_N}$, defined in \cref{pro:2.96}, is provided with the first $N-1$ core tensors left-orthonormal and the last $N$ core tensors right-orthonormal. For an unforced MLTI system, the equilibrium point $\textsf{X}=\textsf{O}$ is asymptotically stable if the largest singular value of
$
\bar{\tilde{\textbf{A}}}^{(N)}
$
is less than one, where $\bar{\tilde{\textbf{A}}}^{(N)}=\texttt{reshape}(\tilde{\textsf{A}}^{(N)},R_{N-1}J_N,R_N)$.
\end{corollary}
\begin{proof}
Based on the results of \cite{Klus_2018}, the singular values of $\bar{\tilde{\textbf{A}}}^{(N)}$ are the singular values of $\varphi(\textsf{A})$. In addition, we know that the magnitude of the maximal eigenvalue of a matrix is less than or equal to its largest singular value. Hence, the proof follows immediately from \cref{pro:3.5}.
\end{proof}

\textit{Remark.}
Although \cref{pro:3.5} offers strong stability results for unforced MLTI systems, computing U-eigenvalues usually requires an order of $\mathcal{O}(\Pi_{\mathcal{J}}^3)$ number of operations through tensor unfolding and matrix eigenvalue decomposition. To the contrary, \cref{cor:5.2,cor:5.3,coro:3.7} can be used to determine the stability of  MLTI systems much faster. In particular, if the TTD of $\tilde{\textsf{A}}$ is provided, the time complexity of left- and right-orthonormalization is about $\mathcal{O}(NJR^3)$ assuming that $J_n=J$ for all $n$, and $R$ is the average of the TT-ranks of $\tilde{\textsf{A}}$ \cite{doi:10.1137/090752286}. Moreover, truncating the TT-rank $\tilde{R}_N$ of $\tilde{\textsf{A}}$ would not alter the largest singular values of $\bar{\tilde{\textbf{A}}}^{(N)}$. Therefore, setting $\tilde{R}_N=1$ and computing the vector 2-norm of $\bar{\tilde{\textbf{A}}}^{(N)}$ will return the largest singular value of $\varphi(\textsf{A})$.

\subsection{Reachability}\label{sec:5.2}
In this and the following subsections, we introduce the definitions of reachability and observability for MLTI systems which are similar to analogous concepts for the LTI systems \cite{doi:10.1137/1.9781611973884, Kailath_1980, Rugh:1996:LST:225486}. We then establish sufficient and necessary conditions for reachability and observability for MLTI systems.

\begin{definition}
 The MLTI system (\ref{eq11}) is said to be reachable on $[t_0,t_1]$ if, given any initial condition $\textsf{X}_0$ and any final state $\textsf{X}_1$, there exists a sequence of inputs $\textsf{U}_t$ that steers the state of the system from $\textsf{X}_{t_0}=\textsf{X}_0$ to $\textsf{X}_{t_1}=\textsf{X}_1$.
\end{definition}

\begin{theorem}\label{thm31}
 The pair $(\textsf{A},\textsf{B})$ is reachable on $[t_0,t_1]$ if and only if the reachability Gramian
\begin{equation}
\textsf{W}_r(t_0,t_1)=\sum_{t=t_0}^{t_1-1}\textsf{A}^{t_1-t-1}*\textsf{B}*\textsf{B}^{\top}*(\textsf{A}^{\top})^{t_1-t-1},
\end{equation}
which is a weakly symmetric even-order square tensor, is U-positive definite.
\end{theorem}
\begin{proof}
Suppose $\textsf{W}_r(t_0,t_1)$ is U-positive definite, and let $\textsf{X}_0$ be the initial state and $\textsf{X}_1$ be the desired final state. Choose $\textsf{U}_t=\textsf{B}^{\top}*(\textsf{A}^{\top})^{t_1-t-1}*\textsf{W}_r^{-1}(t_0,t_1)*\textsf{V}$ for some constant tensor $\textsf{V}$. It follows from the solution of the system (\ref{eq11}) that
$
\textsf{X}_{t_1}= \textsf{A}^{t_1}* \textsf{X}_0+\sum_{j=0}^{t_1-1}\textsf{A}^{t_1-j-1}*\textsf{B}*\textsf{U}_t
= \textsf{A}^{t_1}* \textsf{X}_0+\textsf{W}_r(t_0,t_1)*\textsf{W}_r^{-1}(t_0,t_1)*\textsf{V}
= \textsf{A}^{t_1} *\textsf{X}_0+\textsf{V}.
$
Taking $\textsf{V}=-\textsf{A}^{t_1}* \textsf{X}_0+\textsf{X}_1$, we have $\textsf{X}_{t_1}=\textsf{X}_1$.

We show the converse by contradiction. Suppose $\textsf{W}_r(t_0,t_1)$ is not U-positive definite. Then there exists $\textsf{X}_a\neq \textsf{O}$ such that $\textsf{X}_a^\top*\textsf{A}^{t_1-t-1}*\textsf{B}=\textsf{O}$ for any $t$. Take $\textsf{X}_1=\textsf{X}_a+\textsf{A}^{t_1}*\textsf{X}_0$, and it follows that
$
\textsf{X}_a+\textsf{A}^{t_1}*\textsf{X}_0=\textsf{A}^{t_1}*\textsf{X}_0+\sum_{j=t_0}^{t_1-1}\textsf{A}^{t_1-j-1}*\textsf{B}*\textsf{U}_j\,.
$
Multiplying from the left by $\textsf{X}_a^\top$ yields
$
\textsf{X}_a^\top*\textsf{X}_a=\sum_{j=t_0}^{t_1-1}\textsf{X}_a^\top*\textsf{A}^{t_1-j-1}*\textsf{B}*\textsf{U}_j=0,
$
which implies that $\textsf{X}_a=\textsf{O}$, a contradiction. %So $\mathcal{W}_c(t_0,t_1)$ is U-positive definite, and the proof is thus completed.
\end{proof}

\begin{corollary}
If the reachability Gramian $\textsf{W}_r(t_0,t_1)$ is not M-positive definite, the pair $(\textsf{A},\textsf{B})$ is not reachable on $[t_0,t_1]$. 
\end{corollary}
\begin{proof}
The proof follows immediately from \cref{pro:2.7,thm31}.
\end{proof}

The reachability Gramian assesses to what degree each state is affected by an input \cite{doi:10.1142/S0218127405012429}. The infinite horizon reachability Gramian can be computed from the tensor Lyapunov equation which is defined by
\begin{equation}
\textsf{W}_r - \textsf{A}*\textsf{W}_r*\textsf{A}^\top = \textsf{B}*\textsf{B}^\top.\label{eq:3.6}
\end{equation}
By the unfolding property, if the pair $(\textsf{A},\textsf{B})$ is reachable over an infinite horizon and all the U-eigenvalues of $\textsf{A}$ have magnitude less than one, one can show that there exists a unique weakly symmetric U-positive definite solution $\textsf{W}_r$. Solving the infinite horizon reachability Gramian from the tensor Lyapunov equation may be computationally intensive, so a tensor version of the Kalman rank condition is also provided.

\begin{proposition}\label{coro3.1}
The pair $(\textsf{A},\textsf{B})$ is reachable if and only if the $J_1\times J_1K_1\times \dots \times J_N\times  J_NK_N$ even-order reachability tensor
\begin{equation}\label{eq:3.7}
\mathscr{R}=
\begin{vmatrix}
\textsf{B} & \textsf{A}*\textsf{B} & \dots & \textsf{A}^{\Pi_{\mathcal{J}}-1}*\textsf{B}
\end{vmatrix}
\end{equation}
spans $\mathbb{R}^{J_1\times J_2\times\dots \times J_N}$. In other words, $\text{rank}_U(\mathscr{R})=\Pi_{\mathcal{J}}$.
\end{proposition}
\begin{proof}
The proof follows from \cref{pro:2.9} and the generalized Cayley Hamilton theorem discussed in the tensor eigenvalue problem.
\end{proof}

First, any choice of construction for the mode row block tensor works for the reachability tensor. Second, when $N=1$, \cref{coro3.1} simplifies to the famous Kalman rank condition for reachability of LTI systems. The following corollaries involving with HOSVD (multilinear ranks), CPD (CP rank) and TTD (TT-ranks)  provide useful necessary or sufficient conditions for reachability of MLTI systems if the reachability tensor $\mathscr{R}$ is given in the HOSVD, CPD or TTD format.

\begin{corollary}\label{coro:3.10}
Given the reachability tensor $\mathscr{R}$ in (\ref{eq:3.7}), if $\text{rank}_{2n-1}(\mathscr{R})\neq J_n$ for some $n$, the pair $(\textsf{A},\textsf{B})$ is not reachable.
\end{corollary}
\begin{proof}
The proof follows immediately from \cref{pro:2.95,coro3.1}.
\end{proof}

\begin{corollary}\label{cor:5.9}
Given the reachability tensor $\mathscr{R}$ in (\ref{eq:3.7}), if the set of $n$-mode singular values of $\mathscr{R}$ obtained from the HOSVD contains zero for odd $n$, the pair $(\textsf{A},\textsf{B})$ is not reachable.
\end{corollary}
\begin{proof}
We know that the number of nonvanishing $n$-mode singular values equals to its corresponding $n$-mode multilinear rank. Hence, the result follows immediately from \cref{coro3.1,coro:3.10}.
\end{proof}

\begin{corollary}\label{coro:3.12}
Given the reachability tensor $\mathscr{R}$ in (\ref{eq:3.7}), if the CPD of $\mathscr{R}$ satisfies (\ref{eq:2.17}) with CP rank equal to $\Pi_{\mathcal{J}}$, the pair $(\textsf{A},\textsf{B})$ is reachable. Conversely, if the pair $(\textsf{A},\textsf{B})$ is reachable, then the CP rank of $\mathscr{R}$ is greater than or equal to $\Pi_{\mathcal{J}}$.
\end{corollary}
\begin{proof}
The first part of the proof follows immediately from \cref{pro:2.11,coro3.1}. The second part of the proof follows from the fact that the CP rank of a tensor is greater than or equal to its unfolding rank.
\end{proof}

\begin{corollary}\label{coro:3.13}
Given the reachability tensor $\mathscr{R}$ in (\ref{eq:3.7}), the pair $(\textsf{A},\textsf{B})$ is reachable if and only if the $N$-th optimal TT-rank of $\tilde{\mathscr{R}}\in\mathbb{R}^{J_1\times \dots \times J_N\times J_1K_1\times \dots \times J_NK_N}$, defined in \cref{pro:2.96}, is equal to $\Pi_{\mathcal{J}}$.
\end{corollary}
\begin{proof}
The proof follows immediately from \cref{pro:2.96,coro3.1}.
\end{proof}

\textit{Remark.}
Finding the unfolding rank of  the reachability tensor $\mathscr{R}$ through tensor unfolding and matrix QR decomposition is computationally expensive, and has a $\mathcal{O}(\Pi_{\mathcal{J}}^3\Pi_{\mathcal{K}})$ time complexity. However, if the reachability tensor $\mathscr{R}$  is already given in the tensor decomposition format, computing the unfolding rank can be achieved efficiently based on \cref{cor:5.9,coro:3.12,coro:3.13}. Particularly, if the TTD of $\tilde{\mathscr{R}}$ is provided, we do not need any additional computation to obtain the unfolding rank.

\subsection{Observability}\label{sec:5.3}
The results of observability can be simply obtained by the duality principle, similarly to LTI systems. 
\begin{definition}
The MLTI system (\ref{eq9}) is said to be observable on $[t_0,t_1]$ if any initial state $\textsf{X}_{t_0}=\textsf{X}_0$ can be uniquely determined by $\textsf{Y}_t$ on $[t_0,t_1]$.
\end{definition}

\begin{theorem}\label{thm32}
The pair $(\textsf{A},\textsf{C})$ is observable on $[t_0,t_1]$ if and only if the observability Gramian
\begin{equation}
\textsf{W}_o(t_0,t_1)=\sum_{t=t_0}^{t_1-1}(\textsf{A}^{\top})^{t-t_0}*\textsf{C}^{\top}*\textsf{C}*\textsf{A}^{t-t_0},
\end{equation}
which is a weakly symmetric even-order square tensor, is U-positive definite.
\end{theorem}
%\begin{proof}
%The proof follows similar steps as in Theorem \ref{thm31}. 
%\end{proof}
\begin{proof}
Suppose that $\textsf{W}_o(t_0,t_1)$ is U-positive definite and let $\textsf{X}_0$ be the initial state such that $\textsf{Y}_t = \textsf{C}*\textsf{X}_t=\textsf{C}*\textsf{A}^{t-t_0}*\textsf{X}_0$ for any $t\in[t_0,t_1]$. Multiplying from the left by $(\textsf{A}^\top)^{t-t_0}*\textsf{C}^\top$ yields
$
(\textsf{A}^\top)^{t-t_0}*\textsf{C}^\top*\textsf{Y}_t = (\textsf{A}^{\top})^{t-t_0}*\textsf{C}^{\top}*\textsf{C}*\textsf{A}^{t-t_0}*\textsf{X}_0
$, which implies that
$
\sum_{t=t_0}^{t_1-1} (\textsf{A}^\top)^{t-t_0}*\textsf{C}^\top*\textsf{Y}_t  = \sum_{t=t_0}^{t_1-1}(\textsf{A}^{\top})^{t-t_0}*\textsf{C}^{\top}*\textsf{C}*\textsf{A}^{t-t_0}*\textsf{X}_0
=\textsf{W}_o(t_0,t_1)*\textsf{X}_0.
$
Since $\textsf{W}_o(t_0,t_1)$ is U-invertible, this equation has a unique solution
$
\textsf{X}_0=\textsf{W}_o^{-1}(t_0,t_1)\sum_{t=t_0}^{t_1-1} (\textsf{A}^\top)^{t-t_0}*\textsf{C}^\top*\textsf{Y}_t.
$
Hence, $(\textsf{A},\textsf{C})$ is observable on $[t_0,t_1]$.

Again, we show the converse by contradiction. Suppose that $\textsf{W}_o(t_0,t_1)$ is not U-positive definite. Then there exists $\textsf{X}_a\neq \textsf{O}$ such that $\textsf{C}*\textsf{A}^{t-t_0}*\textsf{X}_a=\textsf{O}$ for any $t$. Take $\textsf{X}_{t_0}=\textsf{X}_0+\textsf{X}_a$ for some initial state $\textsf{X}_0$. Then
$
\textsf{Y}_t = \textsf{C}*\textsf{A}^{t-t_0}*\textsf{X}_0+\textsf{C}*\textsf{A}^{t-t_0}*\textsf{X}_a = \textsf{C}*\textsf{A}^{t-t_0}*\textsf{X}_0
$
for any $t\in[t_0,t_1]$. The initial states $\textsf{X}_0$ and $\textsf{X}_0+\textsf{X}_a$ produce the same output, which implies that $(\textsf{A},\textsf{C})$ is not observable on $[t_0,t_1]$, a contradiction.
\end{proof}

\begin{corollary}
If the observability Gramian $\textsf{W}_o(t_0,t_1)$ is not M-positive definite, the pair $(\textsf{A},\textsf{C})$ is not observable on $[t_0,t_1]$.
\end{corollary}

The observability Gramian assesses to what degree each state affects future outputs \cite{doi:10.1142/S0218127405012429}. The infinite horizon observability Gramian can be computed from the tensor Lyapunov equation defined by
\begin{equation}
\textsf{A}^\top*\textsf{W}_o*\textsf{A} - \textsf{W}_o = -\textsf{C}^\top*\textsf{C}. \label{eq:3.9}
\end{equation}
If the pair $(\textsf{A},\textsf{C})$ is observable and all the U-eigenvalues of $\textsf{A}$ have magnitude less than one, there exists a unique weakly symmetric U-positive definite solution $\textsf{W}_o$.

The following results can be proved similarly to those in \cref{sec:5.2}.

\begin{proposition}\label{pro:3.13}
The pair $(\textsf{A},\textsf{C})$ is observable if and only if the $I_1J_1\times J_1\times \dots \times I_NJ_N\times J_N$ even-order observability tensor
\begin{equation}\label{eq:3.10}
\mathscr{O}=
\begin{vmatrix}
\textsf{C} &
 \textsf{C}*\textsf{A} &
 \dots&
 \textsf{C}*\textsf{A}^{\Pi_{\mathcal{J}}-1}
\end{vmatrix}^\top
\end{equation}
spans $\mathbb{R}^{J_1\times J_2\times\dots \times J_N}$. In other words, $\text{rank}_U(\mathscr{O})=\Pi_{\mathcal{J}}$.
\end{proposition}

\begin{corollary}\label{coro:3.16}
Given the observability tensor $\mathscr{O}$ in (\ref{eq:3.10}), if $\text{rank}_{2n}(\mathscr{O})\neq J_n$ for some $n$, the pair $(\textsf{A},\textsf{C})$ is not observable.
\end{corollary}

\begin{corollary}
Given the observability tensor $\mathscr{O}$ in (\ref{eq:3.10}), if the set of $n$-mode singular values of $\mathscr{O}$ obtained from the HOSVD contains zero for even $n$, the pair $(\textsf{A},\textsf{C})$ is not observable.
\end{corollary}

\begin{corollary}
Given the observability tensor $\mathscr{O}$ in (\ref{eq:3.10}), if the CPD of $\mathscr{O}$ satisfies (\ref{eq:2.17}) with CP rank equal to $\Pi_{\mathcal{J}}$, the pair $(\textsf{A},\textsf{C})$ is observable. Conversely, if the pair $(\textsf{A},\textsf{C})$ is observable, then the CP rank of $\mathscr{O}$ is greater than or equal to $\Pi_{\mathcal{J}}$.
\end{corollary}

\begin{corollary}\label{cor:5.20}
Given the observability tensor $\mathscr{O}$ in (\ref{eq:3.10}), the pair $(\textsf{A},\textsf{C})$ is observable if and only if the $N$-th optimal TT-rank of $\tilde{\mathscr{O}}\in\mathbb{R}^{I_1J_1\times \dots \times I_NJ_N\times J_1\times \dots \times J_N}$, defined in \cref{pro:2.96}, is equal to $\Pi_{\mathcal{J}}$.
\end{corollary}

\section{Model reduction for MLTI systems}\label{sec:6}
Based on the observations in \cref{sec:5}, it is more natural to manipulate MLTI systems in the tensor decomposition format so that all the computational advantages can be realized. This may also result in a more compressed representation.

\subsection{Generalized CPD/TTD}\label{sec:6.1}
We first introduce the notion of generalized CPD/TTD for even-order paired tensors described in \cite{gel2017tensor}, in which the generalized CPD can also be viewed as the extension of the Kronecker rank approximation proposed by Van Loan \cite{loan_2016}. Generalized CPD and TTD share a similar format and possess many analogous properties. 

\begin{definition}\label{def:4.2}
Given an even-order paired tensor $\textsf{A}\in\mathbb{R}^{J_1\times I_1\times\dots \times J_N\times I_N}$, the generalized CPD of  $\textsf{A}$  is given by
\begin{equation}\label{eq:4.15}
\textsf{A}=\sum_{r=1}^R\textsf{A}_{r::}^{(1)}\circ \textsf{A}_{r::}^{(2)}\circ \dots \circ \textsf{A}_{r::}^{(N)},
\end{equation}
where, $\textsf{A}^{(n)}\in\mathbb{R}^{R\times J_n\times I_n}$. Extending Van Loan's definition \cite{loan_2016}, we call the smallest $R$ that achieves (\ref{eq:4.15}) the Kronecker rank of $\textsf{A}$.
\end{definition}

\begin{definition}
Given an even-order paired tensor $\textsf{A}\in\mathbb{R}^{J_1\times I_1\times\dots \times J_N\times I_N}$, the generalized TTD of  $\textsf{A}$  is given by
\begin{equation}\label{eq:gttd}
\textsf{A}=\sum_{r_0=1}^{R_0}\dots \sum_{r_N=1}^{R_N}\textsf{A}^{(1)}_{r_0::r_1}\circ \textsf{A}^{(2)}_{r_1::r_2}\circ \dots \circ \textsf{A}^{(N)}_{r_{N-1}::r_N},
\end{equation}
where, $\textsf{A}^{(n)}\in\mathbb{R}^{R_{n-1}\times J_n\times I_n\times R_n}$, and $\{R_0,R_1,\dots,R_N\}$ is the set of TT-ranks with $R_0=R_N=1$.
\end{definition}

Please refer to \cref{app:5.1} for the use of : notation. Given two even-order paired tensors in the generalized CPD/TTD format, the Einstein product (\ref{eq10}) between the two can be computed without having to reconstruct the full tensors, i.e. keeping the original format \cite{gel2017tensor}. The following proposition states the case for generalized CPD, which also applies to generalized TTD.

\begin{proposition}\label{pro:2.13}
Given two even-order paired tensors $\textsf{A}\in\mathbb{R}^{J_1\times K_1\times \dots J_N\times K_N}$ and $\textsf{B}\in\mathbb{R}^{K_1\times I_1 \times \dots \times K_N\times I_N}$ in the format of (\ref{eq:4.15}) with  Kronekcer ranks $R$ and $S$, respectively, the Einstein product $\textsf{A}*\textsf{B}$ is given by 
\begin{equation}\label{eq:6.3}
\textsf{A}*\textsf{B} = \sum_{t=1}^{T}\textsf{E}^{(1)}_{t::}\circ \textsf{E}^{(2)}_{t::}\circ \dots\circ\textsf{E}^{(N)}_{t::},
\end{equation}
where, $\textsf{E}^{(n)}_{t::} = \textsf{A}^{(n)}_{r::}\textsf{B}^{(n)}_{s::}\in\mathbb{R}^{J_n\times I_n}$, and $t = ivec(\{r,s\},\{R,S\})$ with $T=RS$.
\end{proposition}

\textit{Remark.}
The computational complexity of the Einstein product (\ref{eq:6.3}) is about $\mathcal{O}(NJ^3R^2)$ assuming that $J_n=I_n=K_n=J$ and $R=S$, which is much lower than $\mathcal{O}(J^{3N})$ from the Einstein product (\ref{eq10}) if $R$ is small. 

The generalized CPD can be recovered from the standard CPD, and similarly for generalized TTD (see \cref{alg:5.4}). The algorithm below is extended from the results by Van Loan \cite{loan_2016} about the Kronecker rank approximation. Thus, one can easily obtain generalized CPD by using any technique for computing the standard CPD including alternating least square (ALS) and modified ALS methods \cite{doi:10.1137/07070111X,Kolda06multilinearoperators}.   
 
\begin{algorithm}[h]
\caption{Generalized CPD}
\label{alg:5.3}
\begin{algorithmic}[1]
\STATE{Given an even-order paired tensors $\textsf{A}\in\mathbb{R}^{J_1\times I_1\times \dots \times J_N\times I_N}$}
\STATE{Set $\check{\textsf{A}} = \texttt{reshape}(\textsf{A},J_1I_1,J_2I_2,\dots,J_NI_N)$}
\STATE{Apply CPD algorithms on $\check{\textsf{A}}$} such that $\check{\textsf{A}} = \sum_{r=1}^R \lambda_r \textbf{a}_1^{(r)}\circ \textbf{a}_2^{(r)}\circ\dots \circ\textbf{a}_N^{(r)}$
\STATE{Set $\textsf{A}_{r::}^{(n)} = \lambda_r^{\frac{1}{N}}\texttt{reshape}(\textbf{a}_n^{(r)},J_n,I_n)$} for $n=1,2,\dots,N$
\RETURN Component tensors $\textsf{A}^{(n)}$ for $n=1,2,\dots,N$.
\end{algorithmic}
\end{algorithm}

\subsection{MLTI model reduction}\label{sec:6.2}
The problem of model reduction has been studied heavily in the framework of classical control \cite{DOLGIN2005771, fortuna2012model, obinata2012model}. Methods including proper orthogonal decomposition (POD), scale-separation and averaging, balanced truncation are applied in many engineering applications when dealing with high-dimensional linear/nonlinear systems \cite{Igor}. As mentioned in \cref{sec:3}, using generalized CPD/TTD, we propose a new MLTI representation with fewer parameters. Note that we omit colons in each component tensor in this and the following subsections for simplicity (e.g. $\textsf{A}_{r}^{(n)} = \textsf{A}_{r::}^{(n)}$).

\begin{proposition}\label{pro:4.1}
The MLTI system (\ref{eq11}) is equivalent to 
\begin{equation}\label{eq45}
\begin{cases}
\textsf{X}_{t+1} = \displaystyle\sum_{r=1}^{R_1} \textsf{X}_t\times \{\textsf{A}_{r}^{(1)},\dots,\textsf{A}_{r}^{(N)}\}+\displaystyle\sum_{r=1}^{R_2} \textsf{U}_t\times \{\textsf{B}_{r}^{(1)},\dots,\textsf{B}_{r}^{(N)}\}\\
\textsf{Y}_t  = \displaystyle\sum_{r=1}^{R_3} \textsf{X}_t\times \{\textsf{C}_{r}^{(1)},\dots,\textsf{C}_{r}^{(N)}\}
\end{cases},
\end{equation}
where, $R_1,R_2,R_3$ are the Kronecker ranks of the system, and $\textsf{A}^{(n)}\in\mathbb{R}^{R_1\times J_n\times J_n}$, $\textsf{B}^{(n)}\in\mathbb{R}^{R_2\times J_n\times K_n}$ and $\textsf{C}^{(n)}\in\mathbb{R}^{R_3\times I_n\times J_n}$. 
\end{proposition}
\begin{proof}
The proof follows from \cref{def:4.2,pro:3.2}.
\end{proof}

\textit{Remark.}  The number of parameters of the MLTI system representation (\ref{eq45})  is $R_1\sum_{n=1}^N J_n^2+R_2\sum_{n=1}^N J_nK_n + R_3\sum_{n=1}^N I_nJ_n$. If the Kronecker ranks $R_1,R_2,R_3$ are relatively small, the total number of parameters is much less than that of the MLTI system model (\ref{eq11}) which is given by $\prod_{n=1}^NJ_n^2+\prod_{n=1}^N J_nK_n+\prod_{n=1}^N I_nJ_n$.

The MLTI system representation (\ref{eq45}) is attractive for systems captured by sparse tensors or tensors with low Kronecker ranks where the two advantages, model reduction and computational efficiency, can be exploited. In particular, if \textsf{A}, \textsf{B} and \textsf{C} are fourth-order paired tensors, the generalized CPDs are reduced to matrix SVD problems, see section 9.2 in \cite{loan_2016}. However, there are two major drawbacks. First, for $N>2$, there is no exact method to compute the Kronecker rank of a tensor \cite{doi:10.1137/07070111X}, and truncating the rank does not ensure a good estimate. Second, current CPD algorithms are not numerically stable, which could result in ill-conditioning during the tensor decomposition and low rank approximation. One way to fix these issues is to replace generalized CPD by generalized TTD in (\ref{eq45}), which takes a similar form. The algorithms for computing generalized TTD are numerically stable with unique optimal TT-ranks \cite{doi:10.1137/090752286}. Most importantly, the TTD based results obtained in \cref{sec:5} can be realized in the form of (\ref{eq45}). For example, we can determine the stability of MLTI systems from the TTD of $\tilde{\textsf{A}}$ defined in \cref{pro:2.96}, which can be obtained from the generalized TTD of \textsf{A} efficiently (similar to the remark in \cref{sec:4.3}).

Recall from \cref{sec:3} that one can always convert the MLTI system (\ref{eq11}) to an equivalent LTI form and then apply traditional model reduction approaches, e.g. balanced truncation. However, after converting to a matrix form, the low tensor rank structure exploited in the form of (\ref{eq45})  may not be preserved, and thus low memory requirements cannot be achieved, see \cref{sec:7.4}. Furthermore, as shown in \cite{chen_2019}, the MLTI system (\ref{eq45}) can be used to further develop a higher-order balanced truncation framework directly in the TTD format, which can provide additional computation and memory benefits over unfolding based model reduction methods.

\subsection{Explicit solution and stability}
In addition to using tensor decompositions, we can exploit matrix calculations of the factor matrices $\textsf{A}^{(n)}_r$ to develop notions including explicit solution and stability for the MLTI system (\ref{eq45}) which also have lower computational costs compared to unfolding based methods.

\begin{proposition}[Solution]\label{pro:4.4}
For an unforced MLTI system
$
\textsf{X}_{t+1} = \sum_{r=1}^{R_1} \textsf{X}_t\times\{\textsf{A}_{r}^{(1)},\textsf{A}_{r}^{(2)},\dots,\textsf{A}_{r}^{(N)}\},
$
the solution for $\textsf{X}$ at time $k$, given initial condition $\textsf{X}_0$, is
\begin{equation}\label{eq:4.18}
\textsf{X}_k = \sum_{r=1}^{R_1^k} \textsf{X}_0\times\{\bar{\textsf{A}}_{r}^{(1)},\bar{\textsf{A}}_{r}^{(2)},\dots,\bar{\textsf{A}}_{r}^{(N)}\},
\end{equation}
where, $\bar{\textsf{A}}_{r}^{(n)} = \textsf{A}_{r_1}^{(n)}\textsf{A}_{r_2}^{(n)}\dots\textsf{A}_{r_k}^{(n)}$ for $r=ivec(\{r_1,r_2,\dots,r_k\},\{R_1,R_1,\stackrel{ k}{\cdots},R_1\})$.
\end{proposition}
\begin{proof}
The result follows immediately from \cref{pro:2.13,pro:3.2}.
\end{proof}

If the Kronecker rank $R_1$ is small, computing the explicit solution using (\ref{eq:4.18}) can be faster than using the Einstein product (\ref{eq10}). Additionally, we can assess the stability of the unforced MLTI system of (\ref{eq45}) based upon the Lyapunov approach.

\begin{proposition}[Stability]\label{pro:4.5}
For the unforced MLTI system of (\ref{eq45}), the equilibrium point $\textsf{X}=\textsf{O}$ is 
\begin{enumerate}
\item stable (i.s.L) if $\sum_{r=1}^{R_1}\prod_{n=1}^N\alpha_{r}^{(n)}=1$; \item asymptotically stable (i.s.L) if $\sum_{r=1}^{R_1}\prod_{n=1}^N\alpha_{r}^{(n)}<1$, 
\end{enumerate}
where, $\alpha_{r}^{(n)}$ denote the largest singular values of $\textsf{A}_{r}^{(n)}$.
\end{proposition}
\begin{proof}
Let's consider $V(\textsf{X})=\|\textsf{X}\|$ as the Lyapunov function candidate and let $f(\textsf{X}) = \sum_{r=1}^{R_1}\textsf{X}\times\{\textsf{A}_{r}^{(1)},\textsf{A}_{r}^{(2)},\dots,\textsf{A}_{r}^{(N)}\}$. Then it follows that
$
V(f(\textsf{X}))-V(\textsf{X}) = \|\sum_{r=1}^{R_1}\textsf{X}\times\{\textsf{A}_{r}^{(1)},\textsf{A}_{r}^{(2)},\dots,\textsf{A}_{r}^{(N)}\}\|-\|\textsf{X}\|
\leq \sum_{r=1}^{R_1}\|\textsf{X}\times\{\textsf{A}_{r}^{(1)},\textsf{A}_{r}^{(2)},\dots,\textsf{A}_{r}^{(N)}\}\|-\|\textsf{X}\|
\leq (\sum_{r=1}^{R_1} \prod_{n=1}^N\alpha_{r}^{(n)} - 1)\|\textsf{X}\|,
$
where the last inequality is based on Theorem 6 in \cite{doi:10.1002/nla.2086}. Then the results follow immediately.
\end{proof}

\textit{Remark.} 
The computational complexity of finding the matrix SVDs of the factor matrices can be estimated as $\mathcal{O}(NJ^3R_1)$ assuming that $J_n=J$ for all $n$.

When all the Kronecker ranks of the system $R_1=R_2=R_3=1$, the MLTI system (\ref{eq45}) reduces to the Tucker product representation proposed by Surana et al. \cite{7798500}, which provides a more direct way to see that the Tucker based MLTI model is only a special case of the MLTI system (\ref{eq11}). Additionally, we can obtain stronger stability conditions for the unforced MLTI system in this case.

\begin{proposition}[Stability]\label{pro:4.6}
Suppose that $R_1=1$ in (\ref{eq45}), and $\rho^{(n)}$ are the spectral radii of $\textsf{A}_{1}^{(n)}$. Then the unforced MLTI system of (\ref{eq45}) is
 \begin{enumerate}
 \item stable if and only if $\prod_{n=1}^N\rho^{(n)}\leq1$, and when $\prod_{n=1}^N\rho^{(n)}=1$, their corresponding eigenvalues must have equal algebraic and geometric multiplicity;
 \item asymptotically stable if $\prod_{n=1}^N\rho^{(n)}< 1$;
 \item unstable if $\prod_{n=1}^N\rho^{(n)}> 1$.
 \end{enumerate}
\end{proposition}
\begin{proof}
Based on Equation (2.25) in \cite{doi:10.1137/110820609},
$
\varphi(\textsf{A}) = \textsf{A}_{1}^{(N)}\otimes \textsf{A}_{1}^{(N-1)}\otimes \dots \otimes \textsf{A}_{1}^{(1)}
$
where the operation $\otimes$ denotes the Kronecker product. Moreover, the U-eigenvalues of $\textsf{A}$ are equal to the products of eigenvalues of these component matrices $\textsf{A}_{1}^{(n)}$, and  the U-eigenvalues have equal algebraic and geometric multiplicities if and only if the factor eigenvalues have equal multiplicities \cite{thekroneckerproduct}. Then the results follow immediately from \cref{pro:3.5}.
\end{proof}

The above results including \cref{pro:4.1,pro:4.4,pro:4.5} can be reformulated by replacing the Kronecker rank summation by a series of TT-ranks summations if \textsf{A}, \textsf{B} and \textsf{C} are given in the generalized TTD format. Finally, the Kronecker product can be used to unfold the MLTI system (\ref{eq45}) into a LTI system, i.e. 
\begin{equation*}
\varphi(\textsf{A}) = \sum_{r=1}^{R_1} \textsf{A}_r^{(N)}\otimes\textsf{A}_r^{(N-1)}\otimes \dots \otimes \textsf{A}_r^{(1)},
\end{equation*}
and similarly for $\varphi(\textsf{B})$ and $\varphi(\textsf{C})$. Hence, one can apply traditional control theory techniques to determine the MLTI system properties. 

\section{Numerical examples}\label{sec:7}
We provide four examples to illustrate the MLTI systems theory and model reduction using the techniques developed above. All the numerical examples presented were performed on a Linux machine with 8 GB RAM and a 2.4 GHz Intel Core i5 processor and were conducted in MATLAB R2018a with the Tensor Toolbox 2.6 \cite{tensortoolbox} and the TT toolbox \cite{tttoolbox}.

\subsection{Reachability and observability tensors}
In this example, we consider a simple single-input and single-output (SISO) system that is given by (\ref{eq9}) with
$
\textbf{A}_1=\begin{bmatrix} 0 & 1 & 0\\ 0 & 0 & 1\\0.2 & 0.5 & 0.8\end{bmatrix}, \text{ }\textbf{A}_2=\begin{bmatrix} 0 & 1\\0.5 & 0\end{bmatrix},
\textbf{B}_1 = \begin{bmatrix} 0\\0\\1\end{bmatrix},\text{ } \textbf{B}_2 = \begin{bmatrix} 0 \\1 \end{bmatrix},
\textbf{C}_1 = \begin{bmatrix} 1 & 0 & 0 \end{bmatrix}, \text{ }\textbf{C}_2 = \begin{bmatrix} 1 & 0 \end{bmatrix},
$
and the states $\textsf{X}_t \in \mathbb{R}^{3\times 2}$ are second-order tensors, i.e. matrices. The product of the two spectral radii of $\textbf{A}_1$ and $\textbf{A}_2$ is 0.9207, which implies that the system is asymptotically stable. In addition, the reachability and observability tensors based on (\ref{eq:3.7}) and (\ref{eq:3.10}) are given by
\begin{equation*}
\begin{split}
\mathscr{R}_{::11} &=
\begin{bmatrix}
0 & 0 & 0\\
0 & 1 & 0\\
0 & 0.8 & 0
\end{bmatrix}, \hspace{1.1cm}
\mathscr{R}_{::21} =
\begin{bmatrix}
0 & 0 & 0.5\\
0 & 0 & 0.4\\
1 & 0 & 0.57
\end{bmatrix},\\
\mathscr{R}_{::12} &=
\begin{bmatrix}
0.4 & 0 & 0.378\\
0.57 & 0 & 0.4849\\
0.756 & 0 & 0.6339
\end{bmatrix},
\mathscr{R}_{::22} =
\begin{bmatrix}
0 & 0.285 & 0\\
0 & 0.378 & 0\\
0 & 0.4849 & 0
\end{bmatrix},\\
\end{split}
\end{equation*}
and
\begin{equation*}
\begin{split}
\mathscr{O}_{::11} &=
\begin{bmatrix}
1 & 0 & 0\\
0 & 0 & 0\\
0 & 0 & 0.5
\end{bmatrix}, 
\mathscr{O}_{::21} =
\begin{bmatrix}
0 & 0 & 0\\
0.04 & 0.15 & 0.285\\
0 & 0 & 0
\end{bmatrix},\\
\mathscr{O}_{::12} &=
\begin{bmatrix}
0 & 0 & 0\\
0 & 1 & 0\\
0 & 0 & 0
\end{bmatrix},\hspace{0.3cm}
\mathscr{O}_{::22} =
\begin{bmatrix}
0.1 & 0.25 & 0.4\\
0 & 0 & 0\\
0.057 & 0.1825 & 0.378
\end{bmatrix},\\
\end{split}
\end{equation*}
respectively. We compute the TTDs of the permuted tensors $\tilde{\mathscr{R}}$ and $\tilde{\mathscr{O}}$, respectively and observe that $\text{rank}_U(\mathscr{R})=6$ and $\text{rank}_U(\mathscr{O})=6$. The system therefore is both reachable and observable.

\subsection{Kronecker rank/TT-ranks approximation}\label{sec:7.3}
In this example, we consider a SISO MLTI system (\ref{eq11}) with random sparse tensors $\textsf{A}\in\mathbb{R}^{3\times 3\times 3\times 3\times 3\times 3}$, $\textsf{B}\in\mathbb{R}^{3\times 3\times 3}$ and $\textsf{C}\in\mathbb{R}^{3\times 3\times 3}$. According to \cref{alg:5.3}, we compute the generalized CPDs of \textsf{A}, \textsf{B} and \textsf{C} using the tensor toolbox function \texttt{cp\_als} with estimated Kronecker ranks $R_1=49$, $R_2=2$ and $R_3=2$, respectively, see Generalized CPD in \cref{tab:6.1}. Note that the number of parameters in the system with full Kronekcer ranks could be greater than that for the original system. We then fix $R_2$ and $R_3$ and gradually truncate $R_1$, since $R_1$ is most critical in determining the number of parameters of the reduced system. As we can see in the table, the number of parameters decreases dramatically as $R_1$ decreases. In order to assess the approximation error resulting from this truncation, we compute the relative error using the $\mathcal{H}$-infinity norm $\|\cdot\|_\infty$ between the full system and reduced system transfer functions based on (\ref{eq:3.20}). In particular, we find that when $R_1=10$, the reduced MLTI system is still close to the original system with $\mathcal{H}$-infinity norm relative error of 0.0888.

\begin{figure}[t]\label{fig:6.2}
\centering
\includegraphics[scale=0.32]{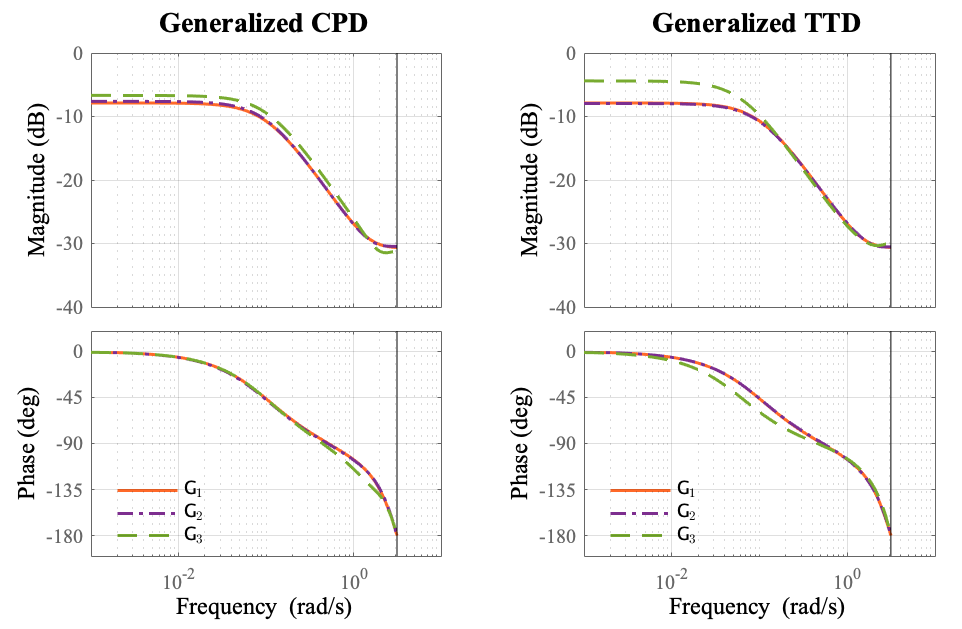}
\caption{\textbf{Bode Diagrams.} $\textsf{G}_1$, $\textsf{G}_2$ and $\textsf{G}_3$ are the transfer functions for the three reduced MLTI systems corresponding to \cref{tab:6.1}, respectively. One may view $\textsf{G}_1$ as the transfer function of the original system. Since the function \texttt{cp\_als} is not numerically stable, the results may not be exactly consistent with Table \ref{tab:6.1} for those obtained by generalized CPD.
} 
\end{figure}

\begin{table}[tbhp]\label{tab:6.1}
\caption{Kronecker rank/TT-ranks approximations of the MLTI system. We omit the first and last trivial TT-ranks in the generalized TTDs of \textsf{A}, \textsf{B} and \textsf{C}.}
{\footnotesize
\begin{center}
\begin{tabular}{|c|c|c|c|} \hline
 & \bf Reduced Ranks &\bf \# Parameters & \bf  $ \frac{\|\textsf{G}_{\text{full}} - \textsf{G}_{\text{red}}\|_{\infty}}{\|\textsf{G}_{\text{full}}\|_\infty}$\\ \hline
 Full System & - & 783 & -\\\hline
Generalized CPD & $\begin{matrix} 49,2,2\\ 20,2,2 \\ 10,2,2\end{matrix}$&$\begin{matrix} 1359\\ 576 \\ 306\end{matrix}$ & $\begin{matrix} 1.58\times 10^{-10}\\ 0.0223 \\ 0.0888\end{matrix}$ \\ \hline
Generalized TTD &$\begin{matrix} \{7,8\},\{1,2\},\{2,2\}\\  \{7,6\},\{1,2\},\{2,2\} \\  \{7,5\},\{1,2\},\{2,2\}\end{matrix}$ & $\begin{matrix} 678\\ 534 \\ 462\end{matrix}$  & $\begin{matrix} 4.39\times 10^{-15}\\ 0.0099 \\ 0.4911\end{matrix}$\\
\hline
\end{tabular}
\end{center}
}
\end{table}

We repeat a similar process for TT-ranks approximation through generalized TTD (see \cref{alg:5.4}). The results are shown in the same table. We find that both generalized CPD and TTD can achieve efficient model reduction while keeping the approximation errors low. Generalized TTD in particular achieves better accuracy for a similar number of reduced parameters as compared to generalized CPD, but the latter can maintain a resonable approximation error with an even lower number of parameters. The Bode diagrams for the reduced MLTI systems are shown in \cref{fig:6.2}. Note that in this example, we manually selected the truncation to study the tradeoff between number of parameters in the reduced system and the approximation error.

\subsection{Memory consumption comparison}\label{sec:7.4}
In this example, we consider a multiple-input and multiple-output (MIMO) MLTI system (\ref{eq11}) with random even-order paired tensors $\textsf{A},\textsf{B},\textsf{C}\in\mathbb{R}^{6\times 6\times 6\times 6\times 6\times 6}$ that possess low TT-ranks. We compare the memory consumptions of the generalized TTD based representation (\ref{eq45}) with the reduced models obtained from the unfolding based balanced truncation. The results are shown in \cref{tab:6.2}. One can clearly  see that if the MLTI systems possess low TT-ranks structure,  the generalized TTD based approach achieves much better accuracy for a similar number of parameters as compared to balanced truncation.

\begin{table}[tbhp]\label{tab:6.2}
\caption{Memory consumption comparison between the generalized TTD and balanced truncation based methods. We reported the TT-ranks of \textsf{A}, \textsf{B} and \textsf{C} (ignoring the first and last trivial TT-ranks) and the number of singular values retained in the Hankel matrix during the balanced truncation.}
{\footnotesize
\begin{center}
\begin{tabular}{|c|c|c|c|} \hline
 & \bf  Ranks &\bf \# Parameters & \bf  $ \frac{\|\textsf{G}_{\text{full}} - \textsf{G}_{\text{red}}\|_{\infty}}{\|\textsf{G}_{\text{full}}\|_\infty}$\\ \hline
  Full System & - & 139968 & -\\\hline
Generalized TTD & \{6,6\}, \{6,6\}, \{6,6\} &  5184 & $3.98\times 10^{-15}$\\ \hline
Balanced Truncation & $\begin{matrix}  200 \\ 100 \\40\end{matrix}$ & $\begin{matrix}  120000 \\ 30000\\ 4800\end{matrix}$ & $\begin{matrix}  0.0169 \\ 0.1001\\ 0.2360\end{matrix}$\\
\hline
\end{tabular}
\end{center}
}
\end{table}

\subsection{Computational time comparison} \label{sec:7.2}
In this example, we consider unforced MLTI systems (\ref{eq45}) with random sparse even-order paired tensors $\textsf{A}\in\mathbb{R}^{2\times 2\times \stackrel{\scriptscriptstyle n}{\cdots} \times 2\times 2}$ in the generalized TTD format such that $\varphi(\textsf{A}) \in\mathbb{R}^{2^n\times 2^n}$. We compare the run time of \cref{coro:3.7} with the matrix SVD of $\varphi(\textsf{A})$ for determining the stability of the systems. The results are shown in \cref{tab:7.2}. When $n \geq 10$, the TTD based method for finding the largest singular value of $\varphi(\textsf{A})$ exhibits a signficant time advantage compared to the matrix SVD based method for which the time increases exponentially. 

\begin{table}[tbhp]\label{tab:7.2}
\caption{Run time comparison between the TTD and SVD based methods in finding the largest singular value of $\varphi(\textsf{A})$. For the TTD based method, we reported computational time includes conversion from the generalized TTD of \textsf{A} to the TTD of $\tilde{\textsf{A}}$ and left- and right-orthonormalization.}

{\footnotesize
\begin{center}
\begin{tabular}{|c|c|c|c|c|c|} \hline
 $\boldsymbol{n}$& \bf TTD(s) &\bf SVD(s) & \bf  $\boldsymbol{\sigma}_{\max}$ & \bf Relative error & \bf Stability\\ \hline
 6 & 0.0399 & $6.8551\times 10^{-4}$& 0.8082 & $1.3738\times 10^{-16} $& asy. stable\\ \hline
 8 & 0.0491 & 0.0439 & 0.9626 & $4.1523\times 10^{-15}$ & asy. stable\\\hline
 10 & 0.0591 & 0.4979 & 0.8645 & $3.8527\times 10^{-15}$ & asy. stable \\\hline
 12 & 0.0909 & 30.7663 & 0.8485 & $5.7573\times 10^{-15}$ & asy. stable \\\hline
 14 & 0.2623 & 2115.1 & 0.9984 & $1.3566\times 10^{-14}$ & asy. stable \\\hline
\end{tabular}
\end{center}
}
\end{table}

\section{Discussion}\label{sec:8}
While tensor unfolding to a matrix form provides the advantage of leveraging highly optimized matrix algebra libraries,  in doing so however one may not be able to exploit the higher-order hidden patterns/structures, e.g. redundancy/correlations, present in the tensor. For instance, in the context of solving PDEs, Brazell et al. \cite{doi:10.1137/100804577} found that higher-order tensor representations preserve low bandwidth, thereby keeping the computational cost and memory requirement low. As shown in \cref{sec:7.2,sec:7.4}, TTD based methods are more efficient in terms of computational speed and memory requirements compared to unfolding based methods when the MLTI systems have low TT-ranks structure. Although CPD typically offers better compression than TTD, the computation of CP rank is NP-hard, and the lower rank approximations can be ill-posed \cite{doi:10.1137/06066518X}. TTD is more suitable for numerical computations with well developed TT-algebra \cite{doi:10.1137/090752286}. Basic tensor operations such as addition, the Einstein product, Frobenius norm, block tensor, solution to multilinear equations and tensor pseudoinverse, can be computed and maintained in the TTD format, without requiring full tensor representation. This can provide significant computational advantages in finding the reachability/observability tensors and associated unfolding ranks according to \cref{coro:3.13,cor:5.20}, and in obtaining solution of the tensor Lyapunov equations. For details, we refer the reader to \cite{chen_2019} and the references therein.

Another line of approach is to exploit the isomorphism property to build algorithms directly in the full tensor format from existing methods. For example, Brazell et al. \cite{doi:10.1137/100804577} proposed Higher-Order Biconjugate Gradient (HOBG) method for solving multilinear systems which can be used for computing U-inverses and MLTI system transfer functions. Analogously, one can generalize the matrix based Rayleigh Quotient Iteration method for computing U-eigenvalues (which can be used for determining MLTI system stability) directly in the tensor form, see \cref{alg:5.2}. However, the computational efficiency of this type of method remains to be investigated. Finally, one can combine tensor algebra based and matrix based methods to provide the advantages of both approaches as hybrid methods, see some examples in \cite{chen_2019} in the context of MLTI model reduction. In future, it would be worthwhile to systematically explore which of the above mentioned approaches or combination thereof is best given the problem structure.

\section{Conclusion}\label{sec:9}

In this paper, we provided a comprehensive treatment of a newly introduced MLTI system representation using even-order paired tensors and the Einstein product. We established new results which enable one to express tensor unfolding based stability, reachability and observability criteria in terms of more standard notions of tensor ranks/decompositions. We introduced a generalized CPD/TTD based model reduction framework which can significantly reduce the number of MLTI system parameters and realize the tensor decomposition based methods. We also presented computational complexity analysis of our proposed framework, and illustrated the benefits through numerical examples. In particular, TTD offers several computational advantages over CPD and HOSVD, and provides a good representational choice for facilitating numerical computations associated with MLTI systems.

As mentioned in \cref{sec:8}, more work is required to fully realize the potential of tensor algebra based computations for MLTI systems.  It will also be worthwhile to develop theoretical and computational framework for observer and feedback control design for MLTI systems, and apply these techniques in real world complex systems. One particular application we plan to investigate is that of cellular reprogramming which involves introducing transcription factors as a control mechanism to transform one cell type to another. These systems naturally have matrix or tensor state spaces describing their genome-wide structure and gene expression \cite{LIU2018,Ronquist11832}. Such applications would also need to account for nonlinearity and stochasticity in tensor based dynamical system representation and analysis framework, and is an important direction for future research.

\appendix
\section{Additional tensor algebra}
\subsection{M-positive definiteness/rank-one positive definiteness}\label{app:1}

\begin{definition}
An even-order square tensor $\textsf{A}\in \mathbb{R}^{J_1\times J_1 \times \dots \times J_N\times J_N}$ is called M-positive definite if the multilinear functional
\begin{equation}\label{eq:2.16}
\textsf{A}\times \{\textbf{x}_1^\top,\textbf{x}_1^\top,\dots,\textbf{x}_N^\top,\textbf{x}_N^\top\}>0,
\end{equation}
for any nonzero vector $\textbf{x}_n$. If all $\textbf{x}_n$ are equal, $\textsf{A}$ is called rank-one positive definite.
\end{definition}

\begin{proposition}\label{pro:2.7}
If an even-order square tensor $\textsf{A} \in \mathbb{R}^{J_1\times J_1 \times \dots \times J_N\times J_N}$ is U-positive definite, it is M-positive definite. Moreover, if $J_1=J_2=\dots =J_N$, U-positive definiteness also implies rank-one positive definiteness.
\end{proposition}
\begin{proof}
By \cref{lem:2.9}, it follows that
$
\textsf{A}\times \{\textbf{x}_1^\top,\textbf{x}_1^\top,\dots,\textbf{x}_N^\top,\textbf{x}_N^\top\}= \textsf{X}^\top *\textsf{A}*\textsf{X}
$
for $\textsf{X} = \textbf{x}_1\circ \textbf{x}_2 \circ \dots \circ \textbf{x}_N$, i.e. $\textsf{X}$ is a rank-one tensor. Moreover, if $J_1=J_2=\dots =J_N$,  M-positive definiteness implies rank-one positive definiteness \cite{Liqun2009ConditionsFS}. Therefore, the results follow immediately. 
\end{proof}

\subsection{Block tensor properties}\label{app:2}
\begin{proposition}\label{pro:2.9}
Let $\textsf{A},\textsf{B}\in\mathbb{R}^{J_1\times I_1\times \dots \times J_N\times I_N}$ and $\textsf{C},\textsf{D}\in \mathbb{R}^{I_1\times K_1\times \dots \times I_N\times K_N}$. Then the following properties hold:
\begin{enumerate}
\item $\mathcal{P}*\begin{vmatrix}
	\textsf{A} & \textsf{B}
\end{vmatrix}_n=\begin{vmatrix}
	\textsf{P}*\textsf{A} & \textsf{P}*\textsf{B}
\end{vmatrix}_n$ for any $\textsf{P}\in \mathbb{R}^{L_1\times J_1\times \dots \times L_N\times J_N}$;
\item $
\begin{vmatrix}
	\textsf{C}\\
	\textsf{D}
\end{vmatrix}_n*\textsf{Q}=
\begin{vmatrix}
	\textsf{C}*\textsf{Q}\\
	\textsf{D}*\textsf{Q}
\end{vmatrix}_n$ for any $\textsf{Q}\in \mathbb{R}^{K_1\times R_1\times \dots \times K_N\times R_N}$;
\item $\begin{vmatrix}
	\textsf{A} & \textsf{B}
\end{vmatrix}_n*
\begin{vmatrix}
	\textsf{C}\\
	\textsf{D}
\end{vmatrix}_n=\textsf{A}*\textsf{C}+\textsf{B}*\textsf{D}$.
\end{enumerate}
\end{proposition}
\begin{proof}
The proof follows immediately from the definition of $n$-mode row/column block tensors and the Einstein product.
\end{proof}

\begin{proposition}\label{pro6}
Let  $\textsf{A},\textsf{B}\in\mathbb{R}^{J_1\times I_1\times \dots \times J_N\times I_N}$ be two even-order paired tensors. Then
$
\varphi(\begin{vmatrix}
	\textsf{A} & \textsf{B}
\end{vmatrix}_n) =
\begin{bmatrix}
	\varphi(\textsf{A}) & \varphi(\textsf{B})
\end{bmatrix}\textbf{P},
$
where $\textbf{P}$ is a column permutation matrix. In particular, when $I_n=1$ for all $n$ or $n=N$, $\textbf{P}$ is the identity matrix.
\end{proposition}
\begin{proof}
We consider the case for $N=2$. Since the size of the odd modes of the block tensor remains the same, we only need to consider the even modes' unfolding transformation. When $n=1$, the index mapping function for the even modes is 
\begin{equation*}
ivec(\textbf{i},\mathcal{I}) = i_1 + 2(i_2-1)I_1,
\end{equation*}
for $i_1 = 1,2,\dots,2I_1$. Based on the definition of $n$-mode row block tensors, the first $I_1$ columns of $\varphi(\begin{vmatrix}
	\textsf{A} & \textsf{B}
\end{vmatrix}_1)$ are the vectorizations of $\textsf{A}_{:i_1:i_2}$ for $i_1=1,2,\dots,I_1$ and $i_2=1$, and the second $I_1$ columns are the vectorizations of $\textsf{B}_{:i_1:i_2}$ for $i_1=I_1+1,I_1+2,\dots,2I_1$ and $i_2=1$. The alternating pattern continues for all $I_2$ pairs of $I_1$ columns. Hence, $\varphi(\begin{vmatrix}
	\textsf{A} & \textsf{B}
\end{vmatrix}_1) =
\begin{bmatrix}
	\varphi(\textsf{A}) & \varphi(\textsf{B})
\end{bmatrix}\textbf{P}$ for some column permutation matrix $\textbf{P}$. 
When $n=2$, the index mapping function for the even modes is given by
\begin{equation*}
ivec(\textbf{i},\mathcal{I}) = i_1 + (i_2-1)I_1,
\end{equation*}
for $i_2=1,2,\dots, 2I_2$. Similarly, the first $I_1I_2$ columns of $\varphi(\begin{vmatrix}
	\textsf{A} & \textsf{B}
\end{vmatrix}_2)$ are the vectorizations of $\textsf{A}_{:i_1:i_2}$ for $i_1=1,2,\dots,I_1$ and $i_2=1,2,\dots,I_2$, and the second $I_1I_2$ columns are the vectorizations of $\textsf{B}_{:i_1:i_2}$ for $i_1=1,2,\dots,I_1$ and $i_2=I_2+1,I_2+2,\dots,2I_2$. Hence, $\varphi(\begin{vmatrix}
	\textsf{A} & \textsf{B}
\end{vmatrix}_2) =
\begin{bmatrix}
	\varphi(\textsf{A}) & \varphi(\textsf{B})
\end{bmatrix}$. A similar analysis can be used to prove the case for $N>2$. Moreover, when $I_n=1$ for all $n$, $\varphi(\textsf{A})$ and $\varphi(\textsf{B})$ are vectors, so no permutation needs to be considered. The proposition can be considered as a special case of Theorem 3.3 in \cite{doi:10.1137/110820609}.
\end{proof}

\section{Tensor ranks/decompositions proofs}\label{app:3}
\subsection{Proof of \cref{pro:2.95}}\label{proof:2.95}
\setcounter{MaxMatrixCols}{20}
\begingroup
\setlength\arraycolsep{2pt}
Without loss of generality, assume that $\Pi_{\mathcal{I}}\leq \Pi_{\mathcal{J}}$ and $\text{rank}_U(\textsf{A})=\Pi_{\mathcal{I}}$. Then $\varphi(\textsf{A})$ has $\Pi_{\mathcal{I}}$ linearly independent columns. The goal here is to construct a transformation from $\varphi(\textsf{A})$ to $\textbf{A}_{(2n)}^\top$, which can be easily visualized through the representation $(z,\mathbb{S})$ defined in (\ref{eq20}). Let
\begin{equation*}
    \begin{split}
        \mathbb{S}_1 & = \footnotesize{\begin{pmatrix}1 & 2 & \dots & N & N+1 & N+2 & \dots & 2N\\1&3&\dots&2N-1&2&4&\dots&2N\end{pmatrix}},\\
        \mathbb{S}_2 & = \footnotesize{\begin{pmatrix}1 & 2 & \dots & 2n-1 & 2n & \dots & 2N-1 & 2N\\1&2&\dots&2n-1&2n+1&\dots&2N&2n\end{pmatrix}},\\
        \mathbb{S}_3 & = \footnotesize{\begin{pmatrix}1& 2 & \dots & N & N+1 & N+2 & \dots & N+n & N+n+1 & \dots & 2N\\1&3&\dots&2N-1&2n&2&\dots&2n-2&2n+2&\dots&2N\end{pmatrix}},\\
        \mathbb{S}_4 & = \footnotesize{\begin{pmatrix}1& 2 & \dots & N & N+1  & \dots & N+n-1 & N+n & \dots & 2N-1 & 2N\\1&3&\dots&2N-1&2&\dots&2n-2&2n+2&\dots&2N&2n\end{pmatrix}}.
    \end{split}
\end{equation*} 
Clearly, $\varphi(\textsf{A})$ and $\textbf{A}_{(2n)}^\top$ can be represented by $(N,\mathbb{S}_1)$ and $(2N-1,\mathbb{S}_2)$, respectively. According to the definition of the index mapping function $ivec(\textbf{i},\mathcal{I})$, we first require a column permutation matrix $\textbf{P}$ such that $\varphi(\textsf{A})\textbf{P}$ is represented by $(N,\mathbb{S}_3)$. Every $I_{n}$ columns of $\varphi(\textsf{A})\textbf{P}$ correspond to the columns of $\textbf{A}_{(2n)}^\top$. Collect each set of $I_{n}$ columns of $\varphi(\textsf{A})\textbf{P}$ and stack them vertically to form a tall matrix $\tilde{\textbf{A}}$ with the representation $(2N-1,\mathbb{S}_4)$. Since the columns of $\varphi(\textsf{A})\textbf{P}$ are linearly independent, $\text{rank}(\tilde{\textbf{A}})=I_n$. Finally, according to the definition of the index mapping function $ivec(\textbf{j},\mathcal{J})$, we require a row permutation matrix $\textbf{Q}$ such that $\textbf{Q}\tilde{\textbf{A}} = \textbf{A}_{(2n)}^\top$. Hence, $\text{rank}_{2n}(\textsf{A})=\text{rank}(\textbf{A}_{(2n)}^\top)=I_n$. Note that the converse of the statement is incorrect.
\endgroup

\subsection{Proof of \cref{pro:2.11}}\label{proof:2.11}
In order to prove \cref{pro:2.11}, we need to introduce the concept of Khatri-Rao product.
\begin{definition}
Given two matrices $\textbf{A}\in  \mathbb{R}^{J\times I}$ and $\textbf{B}\in  \mathbb{R}^{K\times I}$, the Khatri-Rao product, denoted by $\textbf{A}\odot \textbf{B}$, results in a $JK\times I$ matrix:
\begin{equation*}
\textbf{A}\odot \textbf{B}=\begin{bmatrix}
\textbf{a}_{1}\otimes \textbf{b}_{1} & \textbf{a}_{2}\otimes \textbf{b}_{2} & \dots & \textbf{a}_{I}\otimes \textbf{b}_{I}
\end{bmatrix},
\end{equation*}
where, $\otimes$ denotes the Kronecker product, and $\textbf{a}_n$ and $\textbf{b}_n$ are the column vectors of \textbf{A} and \textbf{B}, respectively.
\end{definition}

The following lemma provided by Sidiropoulos et al. \cite{852018,STEGEMAN2007540} gives some properties of rank and $k$-rank of the Khatri-Rao product $\textbf{A}\odot \textbf{B}$.

\begin{lemma} \label{lem:7.2}
Given two matrices $\textbf{A}\in\mathbb{R}^{J\times R},\textbf{B}\in\mathbb{R}^{I\times R}$, the Khatri-Rao product $\textbf{A}\odot\textbf{B}$ has column rank $R$ if
$
k_{\textbf{A}} + k_{\textbf{B}}  \geq R+1
$
for $k_{\textbf{A}}, k_{\textbf{B}}\geq 1$. Moreover, $k_{\textbf{A}\odot\textbf{B}} \geq \min{\{k_{\textbf{A}}+k_{\textbf{B}}-1,R\}}$.
\end{lemma}

\begin{proposition}\label{pro:a3}
Given matrices $\textbf{A}^{(n)}\in\mathbb{R}^{J_n\times R}$, the Khatri-Rao product $\textbf{A}^{(1)}\odot\textbf{A}^{(2)}\odot\dots\odot \textbf{A}^{(N)}$ has column rank $R$ if
$
\sum_{n=1}^N k_{\textbf{A}^{(n)}} \geq R+N-1
$
for $k_{\textbf{A}^{(n)}}\geq 1$.
\end{proposition}
\begin{proof}
Suppose that $N=3$. By \cref{lem:7.2}, the Khatri-Rao product $\textbf{A}^{(1)}\odot\textbf{A}^{(2)}\odot \textbf{A}^{(3)}$ has full column rank $R$ if
$
k_{\textbf{A}^{(1)}\odot\textbf{A}^{(2)}} + k_{\textbf{A}^{(3)}} \geq R+1.
$
Since we know that $k_{\textbf{A}\odot\textbf{B}} \geq \min{\{k_{\textbf{A}}+k_{\textbf{B}}-1,R\}}$, the above inequality can be satisfied if
\begin{equation*}
\min{\{k_{\textbf{A}^{(1)}}+k_{\textbf{A}^{(2)}}-1,R\}} + k_{\textbf{A}^{(3)}} \geq R+1.
\end{equation*}
When $k_{\textbf{A}^{(1)}}+k_{\textbf{A}^{(2)}}> R+1$, the condition is reduced to $k_{\textbf{A}^{(3)}}\geq 1$, and when $k_{\textbf{A}^{(1)}}+k_{\textbf{A}^{(2)}}\leq R+1$, the condition becomes $k_{\textbf{A}^{(1)}}+k_{\textbf{A}^{(2)}} + k_{\textbf{A}^{(3)}} \geq R+2$. Therefore, the Khatri-Rao product $\textbf{A}^{(1)}\odot\textbf{A}^{(2)}\odot \textbf{A}^{(3)}$ has full column rank $R$ if
$
k_{\textbf{A}^{(1)}}+k_{\textbf{A}^{(2)}} + k_{\textbf{A}^{(3)}} \geq R+2.
$
The result can be easily extended to $n=N$ using the same approach.
\end{proof}

Now, we can prove \cref{pro:2.11}. Suppose that $\textsf{A}$ has the CPD format (\ref{eq:8}) with CP rank equal to $R$. Applying the unfolding transformation $\varphi$ yields
\begin{equation*}
\varphi(\textsf{A}) = (\textbf{A}^{(2N-1)}\odot\dots\odot \textbf{A}^{(1)}) \textbf{S} (\textbf{A}^{(2N)}\odot \dots \odot \textbf{A}^{(2)})^\top,
\end{equation*}
where, $\textbf{S}\in\mathbb{R}^{R\times R}$ is a diagonal matrix containing the weights of the CPD on its diagonal. By \cref{pro:a3}, the two Khatri-Rao products $\textbf{A}^{(2N-1)}\odot\dots\odot \textbf{A}^{(1)}$ and $\textbf{A}^{(2N)}\odot\dots \odot \textbf{A}^{(2)}$ have full column rank $R$ if the two conditions
$\sum_{n=1:2}^{2N} k_{\textbf{A}^{(n)}} \geq R + N -1 \text{, and }
\sum_{n=2:2}^{2N} k_{\textbf{A}^{(n)}} \geq R + N - 1
$
are satisfied. Hence, $\text{rank}_U(\textsf{A})= R$. Note that we do not require the CPD of $\textsf{A}$ to be unique in the statement.

\subsection{Proof of \cref{cor:5.3}}\label{app:2.3}
The proof is formulated similarly to the one above. We need to use the properties of Khatri-Rao product.
\begin{lemma}\label{lem:b2}
Given matrices $\textbf{A}^{(n)}\in\mathbb{R}^{J_n\times R}$, the Khatri-Rao product $\textbf{A}^{(1)}\odot\textbf{A}^{(2)}\odot\dots\odot \textbf{A}^{(N)}$ has all the column vectors orthogonal if at least one of $\textbf{A}^{(n)}$ has all the column vectors orthogonal for $n=1,2,\dots,N$.
\end{lemma}
\begin{proof}
Suppose that $N=2$. Based on the properties of Kronecker product, for any $1\leq n,m\leq R$, the inner product between $\textbf{a}_n^{(1)}\otimes \textbf{a}_n^{(2)}$ and $\textbf{a}_m^{(1)}\otimes \textbf{a}_m^{(2)}$ is given by
\begin{equation*}
(\textbf{a}_n^{(1)}\otimes \textbf{a}_n^{(2)})^\top (\textbf{a}_m^{(1)}\otimes \textbf{a}_m^{(2)}) = ((\textbf{a}_n^{(1)})^\top \textbf{a}_m^{(1)}))\otimes ((\textbf{a}_n^{(2)})^\top \textbf{a}_m^{(2)})).
\end{equation*}
Therefore, if $\textbf{A}^{(1)}$ or $\textbf{A}^{(2)}$ has all column vectors orthogonal, then the inner product between $\textbf{a}_n^{(1)}\otimes \textbf{a}_n^{(2)}$ and $\textbf{a}_m^{(1)}\otimes \textbf{a}_m^{(2)}$ is zero for any $n,m$. 
\end{proof}

Now we can prove \cref{cor:5.3}. Suppose that $\textsf{A}$ has the CPD format (\ref{eq:8}). Applying the unfolding transformation $\varphi$ yields
\begin{equation*}
\varphi(\textsf{A}) = (\textbf{A}^{(2N-1)}\odot\dots\odot \textbf{A}^{(1)}) \textbf{S} (\textbf{A}^{(2N)}\odot \dots \odot \textbf{A}^{(2)})^\top,
\end{equation*}
where, $\textbf{S}\in\mathbb{R}^{R\times R}$ is a diagonal matrix containing the weights of the CPD on its diagonal. By \cref{lem:b2}, the two Khatri-Rao products $\textbf{A}^{(2N-1)}\odot\dots\odot \textbf{A}^{(1)}$ and $\textbf{A}^{(2N)}\odot\dots \odot \textbf{A}^{(2)}$ have all the column vectors orthonormal if $\textbf{A}^{(n)}$ and $\textbf{A}^{(m)}$ have all the column vectors orthonormal for at least one odd $n$ and even $m$. Thus, $\lambda_1$ will be the largest singular value of $\varphi(\textsf{A})$. In addition, we know that the magnitude of the maximal eigenvalue of a matrix is less than or equal to its largest singular value. Hence, the proof follows immediately from \cref{pro:3.5}. Note that there is one special case when the CPD uniqueness condition fails, i.e. $\sum_{n=1}^{2N}k_{\textbf{A}^{(n)}}=2R+2N-2$. However, different CPDs, satisfying the orthonormal condition, correspond to the same matrix SVD under $\varphi$ up to some orthogonal transformations.

\section{Numerical algorithms}\label{app:4}

\begin{algorithm}[h!]
\caption{Generalized TTD}
\label{alg:5.4}
\begin{algorithmic}[1]
\STATE{Given an even-order paired tensors $\textsf{A}\in\mathbb{R}^{J_1\times I_1\times \dots \times J_N\times I_N}$}
\STATE{Set $\check{\textsf{A}} = \texttt{reshape}(\textsf{A},J_1I_1,J_2I_2,\dots,J_NI_N)$}
\STATE{Apply the standard TTD algorithm on $\check{\textsf{A}}$} such that 
\begin{equation*}
\check{\textsf{A}} = \sum_{r_0=1}^{R_0}\dots\sum_{r_N=1}^{R_N} \check{\textsf{A}}_{r_0:r_1}^{(1)}\circ \check{\textsf{A}}_{r_1:r_2}^{(2)}\circ\dots \circ\check{\textsf{A}}_{r_{N-1}:r_N}^{(N)}
\end{equation*}
\STATE{Set $\textsf{A}_{r_{n-1}::r_n}^{(n)} = \texttt{reshape}(\check{\textsf{A}}_{r_{n-1}:r_n}^{(n)},J_n,I_n)$} for $n=1,2,\dots,N$
\RETURN Component tensors $\textsf{A}^{(n)}$ for $n=1,2,\dots,N$.
\end{algorithmic}
\end{algorithm}

\begin{algorithm}[h!]
\caption{Higher-Order Rayleigh Quotient Iteration}
\label{alg:5.2}
\begin{algorithmic}[1]
\STATE{Given an even-order square tensor $\textsf{A}\in\mathbb{R}^{J_1\times J_1\times \dots \times J_N\times J_N}$}\\
\STATE{Initialize $\textsf{X}_0\in\mathbb{R}^{J_1\times J_2\times \dots \times J_N}$ with $\|\textsf{X}_0\|=1$} \\
\STATE{Compute $\lambda_0=\textsf{X}_0^\top*\textsf{A}*\textsf{X}_0$}\\
\FOR{$k=1,2,\dots$}
\STATE{Solve $(\textsf{A}-\lambda_{k-1}\textsf{I})*\textsf{Y} = \textsf{X}_{k-1}$ using HOBG proposed in \cite{doi:10.1137/100804577}}\\
\STATE{Set $\textsf{X}_k = \frac{\textsf{Y}}{\|\textsf{Y}\|}$}\\
\STATE{Compute $\lambda_k = \textsf{X}_k^\top*\textsf{A}*\textsf{X}_k$}
\ENDFOR
\RETURN U-eigenvalue $\lambda$ and U-eigentensor $\textsf{X}$.
\end{algorithmic}
\end{algorithm}

\vspace{1cm}
\section{MATLAB functions}\label{app:5}
\subsection{The colon operator} \label{app:5.1}
The colon : is one of the most useful operators in MATLAB, which can create vectors, subscript arrays and specify for iterations. For our purpose, it acts as shorthand to include all subscripts in a particular array dimension \cite{MATLAB:2018}. For example, $\textbf{A}_{:i}$ is equivalent to $\textbf{A}_{ji}$ for all $j$. In the following, we represent TTD, generalized CPD and TTD in the component-wise form. 
\begin{enumerate}
\item (\ref{train}) $\Leftrightarrow \textsf{X}_{j_1j_2\dots j_N} = \sum_{r_0=1}^{R_0}\dots \sum_{r_N=1}^{R_N}\textsf{X}^{(1)}_{r_0j_1r_1}  \textsf{X}^{(2)}_{r_1j_2r_2} \dots  \textsf{X}^{(N)}_{r_{N-1}j_Nr_N}$.
\item (\ref{eq:4.15}) $\Leftrightarrow \textsf{A}_{j_1i_1\dots j_Ni_N}=\sum_{r=1}^R\textsf{A}_{rj_1i_1}^{(1)} \textsf{A}_{rj_2i_2}^{(2)} \dots  \textsf{A}_{rj_Ni_N}^{(N)}$.
\item (\ref{eq:gttd}) $\Leftrightarrow\textsf{A}_{j_1i_1\dots j_Ni_N}=\sum_{r_0=1}^{R_0}\dots \sum_{r_N=1}^{R_N}\textsf{A}^{(1)}_{r_0j_1i_1r_1} \textsf{A}^{(2)}_{r_1j_2i_2r_2} \dots  \textsf{A}^{(N)}_{r_{N-1}j_Ni_Nr_N}$.
\end{enumerate}

\subsection{The \texttt{reshape} operator}\label{app:5.2}
The command $\textsf{B}=\texttt{reshape}(\textsf{A},J_1,J_2,\dots,J_N)$ reshapes a tensor \textsf{A} into a $J_1\times J_2\times \dots \times J_N$ order tensor such that the number of elements in \textsf{B} matches the number of elements in \textsf{A} \cite{MATLAB:2018}.

\section*{Acknowledgments} We would like to thank Dr. Frederick Leve at the Air Force Office of Scientific Research (AFOSR) for support and encouragement. We would also like to thank the two referees for their constructive comments, which lead to a significant improvement of the paper. 

\bibliographystyle{siamplain}
\bibliography{reference}

\end{document}